\RequirePackage{fix-cm}
\RequirePackage{fixltx2e}
\documentclass[english]{paper}
\usepackage[T1]{fontenc}
\usepackage[latin9]{inputenc}
\usepackage{geometry}
\usepackage{color}
\usepackage{babel}
\usepackage{algorithm2e}
\usepackage{amsmath}
\usepackage{amsthm}
\usepackage{amssymb}
\usepackage{setspace}
\usepackage[unicode=true,pdfusetitle,
 bookmarks=true,bookmarksnumbered=false,bookmarksopen=false,
 breaklinks=false,pdfborder={0 0 0},backref=false,colorlinks=true]
 {hyperref}
\hypersetup{
 citecolor=blue, urlcolor=blue, linkcolor=red, }

\makeatletter
\usepackage{enumitem}		
  \theoremstyle{plain}
  \newtheorem*{thm*}{\protect\theoremname}
  \theoremstyle{remark}
  \newtheorem{rem}{\protect\remarkname}[section]
  \theoremstyle{plain}
  \newtheorem*{cor*}{\protect\corollaryname}
  \theoremstyle{definition}
  \newtheorem{defn}{\protect\definitionname}[section]
  \theoremstyle{plain}
  \newtheorem{thm}{\protect\theoremname}[section]
  \theoremstyle{definition}
  \newtheorem*{example*}{\protect\examplename}
  \theoremstyle{remark}
  \newtheorem*{rem*}{\protect\remarkname}
  \theoremstyle{plain}
  \newtheorem{prop}{\protect\propositionname}[section]
  \theoremstyle{plain}
  \newtheorem{lem}{\protect\lemmaname}[section]
  \theoremstyle{plain}
  \newtheorem{cor}{\protect\corollaryname}[section]

\usepackage{amssymb, latexsym, amsmath, amscd, epsfig, color}

\newtheorem*{thmdefn}{Definition/Theorem}

\newcommand{\real}{\mathbb{R}}
\newcommand{\hy}{\mathbb{H}}
\newcommand{\psl}{\mathrm{PSL}}
\newcommand{\co}{\mathbb{C}}

\newcommand{\vbdy}{\partial_{\infty}}

\newcommand{\us}{\widetilde{\Sigma}}
\newcommand{\limitset}{\Lambda{(\Gamma)}}
\newcommand{\ch}{\mathrm{Conv}(\Lambda{(\Gamma))}}
\newcommand{\Isom}{\mathrm{Isom}}

\makeatother

  \providecommand{\corollaryname}{Corollary}
  \providecommand{\definitionname}{Definition}
  \providecommand{\examplename}{Example}
  \providecommand{\lemmaname}{Lemma}
  \providecommand{\propositionname}{Proposition}
  \providecommand{\remarkname}{Remark}
  \providecommand{\theoremname}{Theorem}
\providecommand{\corollaryname}{Corollary}
\providecommand{\theoremname}{Theorem}

\begin{document}

\title{Entropy, Critical Exponent and Immersed Surfaces in Hyperbolic 3-Manifolds}

\author{Lien-Yung Kao\thanks{Department of Mathematics, University of Notre Dame, Notre Dame, IN
46545 USA. \textit{E-mail}\texttt{:\protect\href{mailto:lkao@nd.edu}{lkao@nd.edu}}}}

\maketitle
\begin{abstract}
\textsf{We consider a $\pi_{1}$--injective immersion $f:\Sigma\to M$
from a compact surface $\Sigma$ to a hyperbolic 3--manifold $M$.
Let $\Gamma$ denote the copy of $\pi_{1}\Sigma$ in $\mathrm{Isom}(\hy{}^{3})$
induced by the immersion and $\delta(\Gamma)$ be the critical exponent.
Suppose $\Gamma$ is convex cocompact and $\Sigma$ is negatively
curved, we prove that there are two geometric constants $C_{1}(\Sigma,M)$
and $C_{2}(\Sigma,M)$ not bigger than $1$ such that $C_{1}(\Sigma,M)\cdot\delta_{\Gamma}\leq h(\Sigma)\leq C_{2}(\Sigma,M)\cdot\delta_{\Gamma}$,
where $h(\Sigma)$ is the topological entropy of the geodesic flow
on. When $f$ is an embedding, we show that $C_{1}(\Sigma,M)$ and
$C_{2}(\Sigma,M)$ are exactly the geodesic stretches (a.k.a. Thurston's
intersection number) with respect to certain Gibbs measures. Moreover,
we prove the rigidity phenomenon arising from this inequality. Lastly,
as an application, we discuss immersed minimal surfaces in hyperbolic
3--manifolds and these discussions lead us to results similar to A.
Sanders' work \cite{Sanders:2014wv} on the moduli space of $\Sigma$
introduced by C. Taubes \cite{Taubes:2004ke}.}
\end{abstract}

\begin{doublespace}
\tableofcontents{}
\end{doublespace}

\newpage{}

\section{Introduction}

\subsection{Main results}

\textsf{We consider a $\pi_{1}$--injective immersion $f:\Sigma\to M$
from a compact surface $\Sigma$ to a hyperbolic 3--manifold $M$.
Let $\Gamma$ denote the copy of $\pi_{1}\Sigma$ in $\mathrm{Isom}(\hy^{3})$
induced by the immersion $f$, and we endow $\Sigma$ with the induced
metric $g$ from the given hyperbolic metric $h$ on $M$. The topological
entropy $h(\Sigma)$ of the geodesic flow on $T^{1}\Sigma$, and the
critical exponent $\delta_{\Gamma}$ of $\Gamma$ on $\hy^{3}$ are
two natural geometric quantities associated to this setting. Recall
that when $(\Sigma,g)$ is a negatively curved manifold, then each
closed geodesic on $\Sigma$ corresponds to a unique conjugacy class
$[\gamma]\in[\pi_{1}\Sigma]$, and vice versa. We can write the topological
entropy of the geodesic flow on $T^{1}\Sigma$ as 
\[
h(\Sigma)=\lim_{T\to\infty}\frac{1}{T}\log\#\{[\gamma]\in[\pi_{1}S];l_{g}(\gamma)\leq T\},
\]
where $l_{g}(\gamma)$ is the length of the closed geodesic $[\gamma]$
with respect to the metric $g.$ Moreover, the critical exponent $\delta_{\Gamma}$
can be understood by lengths of closed geodesics as well. By Sullivan's
theorem \cite{Sullivan:1979vs}, when $\Gamma$ is convex cocompact,
we have 
\[
\delta_{\Gamma}=\lim_{T\to\infty}\frac{1}{T}\log\#\{[\gamma]\in[\pi_{1}\Sigma];l_{h}(\gamma)\leq T\},
\]
where $l_{h}(\gamma)$ is the length of $[\gamma]$ using the hyperbolic
metric $h.$}

\textsf{The main tool used in the note will be the Thermodynamic Formalism.
Specifically, the reparametrization method introduced by Ledrappier
\cite{Ledrappier:1994uy} and Sambarino \cite{Sambarino:2014jv}.
Through the reparametrization method, we can link two different Anosov
flows on $\hy^{3}$ by a Hölder continuous function. Thus, we can
compare periods of closed orbits associating with different Anosov
flows, therefore, their topological entropies.}

\textsf{Our main theorem shows that we can relate the two geometric
quantities $h(\Sigma)$ and $\delta_{\Gamma}$ by an inequality. Moreover,
the equality cases exhibit rigidity features. }

\begin{thm*}
[Theorem \ref{Thm: main-1}]Let $f:\Sigma\to M$ be a $\pi_{1}-$injective
immersion from a compact surface $\Sigma$ to a hyperbolic 3-manifold
$M$, and $\Gamma$ be the copy of $\pi_{1}\Sigma$ in $\mathrm{Isom}(\hy^{3})$
induced by the immersion $f$. Suppose $\Gamma$ is convex cocompact
and $(\Sigma,f^{*}h)$ is negatively curved, then 
\begin{equation}
C_{1}(\Sigma,M)\cdot\delta_{\Gamma}\leq h(\Sigma)\leq C_{2}(\Sigma,M)\cdot\delta_{\Gamma},\label{eq:ineq}
\end{equation}
where $C_{1}(\Sigma,M)$ and $C_{2}(\Sigma,M)$ are two geometric
constants not bigger than 1. Moreover, each equality holds if and
only if the marked length spectrum of $\Sigma$ is proportional to
the marked length spectrum of $M$, and the proportion is the ratio
$\frac{\delta_{\Gamma}}{h(\Sigma)}.$ 
\end{thm*}

\begin{rem}
:

\begin{itemize}

\item[1.] By Sullivan's theorem \cite{Sullivan:1984bm}, one can
replace the critical exponent $\delta_{\Gamma}$ in $(\ref{eq:ineq})$
by the Hausdorff dimension $\dim_{H}\Lambda(\Gamma)$ of the limit
set $\limitset$. 

\item[2.] This result should be compared with Theorem 1 \cite{Burger:1993wb},
Theorem 1.2 \cite{Anonymous:1995bn}, and Theorem A \cite{Sambarino:2013vp}.
All these results possess a similar flavor of comparing entropies. 

\item[3.] In Glorieux's thesis \cite{Glorieux:2015uo}, he follows
Knieper's method and deduces an upper bound of $h(\Sigma)$ in the
case that $\Sigma$ is embedded in a quasi-Fuchsian manifold $M$.
We will prove that the upper bound in Glorieux's thesis is exactly
the same as the one in Theorem \ref{Thm: main-1}.

\end{itemize}
\end{rem}

\textsf{Next two theorems depict geometric meanings of $C_{1}(\Sigma,M)$
and $C_{2}(\Sigma,M)$ mentioned in Theorem \ref{Thm: main-1}. These
two constants could be regarded as averages of lengths of closed geodesics
with respect to different metrics $g$ and $h$.}

\begin{thm*}
[Theroem \ref{thm: c1c2}]Let $f:\Sigma\to M$ be a $\pi_{1}-$injective
immersion from a compact surface $\Sigma$ to a hyperbolic 3-manifold
$M$, and $\Gamma$ be the copy of $\pi_{1}\Sigma$ in $\mathrm{Isom}(\hy^{3})$
induced by the immersion $f$. Suppose $\Gamma$ is convex cocompact
and $(\Sigma,f^{*}h)$ is negatively curved, then 
\[
C_{2}(\Sigma,M)=\lim_{T\to\infty}\frac{{\displaystyle \sum_{[\gamma]\in R_{T}(g)}l_{h}(\gamma)}}{{\displaystyle \sum_{[\gamma]\in R_{T}(g)}l_{g}(\gamma)}};\quad C_{1}(\Sigma,M)=\lim_{T\to\infty}\frac{{\displaystyle \sum_{[\gamma]\in R_{T}(h)}l_{h}(\gamma)}}{{\displaystyle \sum_{[\gamma]\in R_{T}(h)}l_{g}(\gamma)}}
\]
 where 
\[
R_{T}(g):=\{[\gamma]\in[\pi_{1}\Sigma]:\mbox{ }l_{g}(\gamma)\leq T\},\mbox{and }R_{T}(h):=\{[\gamma]\in[\pi_{1}\Sigma]:\mbox{ }l_{h}(\gamma)\leq T\}.
\]

\end{thm*}

\textsf{In additional, when $f:\Sigma\to M$ is an embedding, we have
another geometrical interpretation of $C_{1}(\Sigma,M)$ and $C_{2}(\Sigma,M)$
. The following theorem shows that $C_{1}(\Sigma,M)$ and $C_{2}(\Sigma,M)$
are the geodesic stretches of $\Sigma$ relative to $M$ with respect
to certain Gibbs measures.}

\begin{thm*}
[Theorem \ref{Thm: geodesic stetch}]Let $f:\Sigma\to M$ be a $\pi_{1}-$injective
embedding from a compact surface $\Sigma$ to a hyperbolic 3-manifold
$M$, and $\Gamma$ be the copy of $\pi_{1}\Sigma$ in $\mathrm{Isom}(\hy^{3})$
induced by the embedding $f$. Suppose $\Gamma$ is convex cocompact
and $(\Sigma,f^{*}h)$ is negatively curved, then
\begin{alignat*}{1}
C_{1}(\Sigma,M) & =I_{\mu}(\Sigma,M),\\
C_{2}(\Sigma,M) & =I_{\mu_{BM}}(\Sigma,M).
\end{alignat*}

Here, $I_{\mu}(\Sigma,M)$ and $I_{\mu_{BM}}(\Sigma,M)$ are the geodesic
stretches with respect to a Gibbs measure $\mu$ and the Bowen-Margulis
measure $\mu_{BM}$ of the geodesic flow $\phi$ on $T^{1}\Sigma$. 
\end{thm*}

\begin{rem}
Our definition of the geometric stretch $I_{\mu}(\Sigma,M)$ in Section
\ref{sec:Geodesic-Stretch} is inspired by Knieper \cite{Anonymous:1995bn}.
In the general setting, the geodesic stretch was introduced in the
paper \cite{Croke:1990eb} of Croke and Fathi is known as the Thurston's
intersection number.
\end{rem}

\subsection{Applications}

\textsf{By the Gauss equation, immersed minimal surfaces in hyperbolic
3--manifolds are negatively curved. Minimal surfaces in hyperbolic
3--manifolds is a very rich subject and have drawn a lot of attention,
with important contributions by Uhlenbeck \cite{Uhlenbeck:1983wl}
and Taubes \cite{Taubes:2004ke}. In this note, we take a glance at
this rich subject from a dynamical system point of view. }

\textsf{The following corollary is a consequence of the main theorem. }

\begin{cor*}
[Corollary \ref{Cor: main}]Let $f:\Sigma\to M$ be a $\pi_{1}$--injective
minimal immersion from a compact surface $\Sigma$ to a hyperbolic
3--manifold $M$, and $\Gamma$ be the copy of $\pi_{1}\Sigma$ in
$\mathrm{Isom}(\hy^{3})$ induced by the immersion. Suppose $\Gamma$
is convex cocompact, then there are explicit constants $C_{1}(\Sigma,M)$and
$C_{2}(\Sigma,M)$not bigger than 1 such that 
\[
C_{1}(\Sigma,M)\cdot\delta_{\Gamma}\leq h(\Sigma)\leq C_{2}(\Sigma,M)\cdot\delta_{\Gamma}
\]
Moreover, each equality holds if and only if the marked length spectrum
of $\Sigma$ is proportional to the marked length spectrum of $M$,
and the proportion is the ration $\frac{\delta_{\Gamma}}{h(\Sigma)}.$ 
\end{cor*}

\begin{rem}
:

From \cite{Uhlenbeck:1983wl}, we learn that the $\pi_{1}$--injectivity
is guaranteed if we put some curvature conditions on $\Sigma$. Namely,
all principal curvatures are between $-1$ and $1$. Furthermore,
in such cases, immersed minimal surfaces are indeed embedded. Therefore,
we can interpret the constants $C_{i}(\Sigma,M)$ of such pairs $(\Sigma,M)$
as geodesic stretches. 
\end{rem}

\textsf{In the last part of this note, we change gear to the Taubes'
moduli space of $\Sigma.$ Taubes \cite{Taubes:2004ke} constructs
the space of minimal hyperbolic germs $\mathcal{H}$ which is a deformation
space for the set whose archetypal elements is a pair that consists
of a Riemannian metric $g$ and the second fundamental form $B$ from
a closed, oriented, negative Euler characteristic minimal surfaces
$\Sigma$ in some hyperbolic 3--manifold $M$. }

\textsf{Uhlenbeck \cite{Uhlenbeck:1983wl} proved that there exists
a representation $\rho:\pi_{1}(\Sigma)\to\mbox{Isom }(\hy{}^{3})\cong\psl(2,\co)$
leaving this minimal immersion invariant. In other words, there is
a map 
\[
\Phi:\mathcal{H}\to\mathcal{R}(\pi_{1}(\Sigma),\psl(2,\co)),
\]
where $\mathcal{R}(\pi_{1}(\Sigma),\psl(2,\co))$ is the space of
conjugacy classes of representations of $\pi_{1}(S)$ into $\psl(2,\co)$.}

\textsf{The following corollary gives an upper and a lower bound of
the topological entropy $h(g,B)$ of the geodesic flow on $T^{1}\Sigma$
provided the data $(g,B)\in\mathcal{H}$. }

\begin{cor*}
[Corollary \ref{Cor: taubes' space}]Let $\rho\in\mathcal{R}(\pi_{1}(\Sigma),\mathrm{PSL}(2,\mathbb{C}))$
be a discrete, faithful and convex cocompact representation. Suppose
$(g,B)\in\Phi^{-1}(\rho)$, then there are explicit constants $C_{1}(g,B)$and
$C_{2}(g,B)$not bigger than 1 such that 
\[
C_{1}(g,B)\cdot\delta_{\rho(\pi_{1}\Sigma)}\leq h(g,B)\leq C_{2}(g,B)\cdot\delta_{\rho(\pi_{1}\Sigma)}\leq\delta_{\rho(\pi_{1}\Sigma)}
\]
with the last equality if and only if $B$ is identically zero which
holds if and only if $\rho$ is Fuchsian. 
\end{cor*}

\begin{rem}
:

\begin{itemize}

\item[1.] When $\rho$ is quasi-Fuchsian, the above upper bound of
$h(g,B)$ is a special case treated in Sanders' paper \cite{Sanders:2014wv}.
However, the inequality in Sanders' work has no information about
the constant $C_{2}(g,B)$. 

\item[2.] The lower bound given in Sanders' paper \cite{Sanders:2014wv}
is derived by Manning's formula \cite{Manning:1981es}, which gives
a lower bound of the topological entropy $h(g,B)$ in terms of the
curvature of $(\Sigma,g).$ Whereas, our method doesn't see the curvature
directly. It will be interesting to compare our lower bound with Sanders'
lower bound of $h(g,B)$.

\end{itemize}
\end{rem}
\hspace{-0.5cm}

\textsf{From above corollary, we recover the Bowen's rigidity theorem
\cite{Bowen:1979eh}.}

\begin{cor*}
[Bowen's rigidity \cite{Bowen:1979eh}]A quasi-Fuchsian representation
$\rho\in\mathcal{QF}$ is Fuchsian if and only if $\dim_{H}\Lambda(\Gamma)=1$.
\end{cor*}

\textsf{Lastly we focus on a two special subsets of the minimal hyperbolic
germs: the Fuchsian space $\mathcal{F}$ and the almost-Fuchsian space
$\mathcal{AF}$. The Fuchsian space is the space of all hyperbolic
metrics on $\Sigma$, i.e. 
\[
\mathcal{F}=\{(m,0)\in\mathcal{H};\mbox{ }m\mbox{ is a hyperboilic metric on }\Sigma\}
\]
and the almost-Fuchsian space $\mathcal{AF}$ is defined by the condition
that $\left\Vert B\right\Vert _{g}^{2}<2$. By Uhlenbeck's theorem
in \cite{Uhlenbeck:1983wl} we know that if $(g,B)\in\mathcal{AF}$
then there exists a unique quasi-Fuchsian 3--manifold $M$, up to
isometry, such that $\Sigma$ is an embedded minimal surface in $M$
with the induced metric $g$ and the second fundamental form $B$. }

\textsf{The following theorem was first proved in \cite{Sanders:2014wv},
and we recover this theorem by the reparametrization method. }

\begin{cor*}
[Theorem \ref{Thm: entropy decreasing}, Theorem 3.5  \cite{Sanders:2014wv}]Consider
the entropy function restricting on the almost-Fuchsian space $h:\mathcal{AF}\to\real$,
then \begin{itemize}

\item[1.] the entropy function $h$ realizes its minimum at the Fuchsian
space $\mathcal{F}$, and

\item[2.] for $(m,0)\in\mathcal{F}$, $h$ is monotone increasing
along the ray $r(t)=(g_{t},tB)$ provided $\left\Vert tB\right\Vert _{g_{t}}<2$,
i.e. $r(t)\subset\mathcal{AF}$, where $g_{t}=e^{2u_{t}}m$.

\end{itemize}
\end{cor*}

\textsf{In the end this note, we give another proof of the following
theorem given in \cite{Sanders:2014wv}. Similar to Bridgeman's method
in \cite{Bridgeman:2010kt}, the pressure form is used in showing
that the Hessian of the entropy defines a metric on $\mathcal{F}$,
which is bounded below by the Weil-Petersson metric. }

\begin{thm*}
[Theorem \ref{thm:hessian entropy}, Theorem 3.8  \cite{Sanders:2014wv} ]One
can define a Riemannian metric on the Fuchsian space $\mathcal{F}$
by using the Hessian of $h$. Moreover, this metric is bounded below
by $2\pi$ times the Weil-Petersson metric on $\mathcal{F}$.
\end{thm*}

\textsf{It is natural to ask if this metric is different from the
Weil-Petersson metric. It will be interesting to learn more relations
between this metric coming from the Hessian of the entropy and the
Weil-Petersson metric or pressure metric on the quasi-Fuchsian spaces. }

\subsubsection*{Acknowledgments}

\textsf{The author is extremely grateful to his Ph.D advisor Dr. François
Ledrappier. This work would have never been possible without François'
support, guidance, patience, and sharing of his insightful ideas.
The author also would like to thank Olivier Glorieux and Andy Sanders
for useful discussions on their works. }

\section{Preliminaries\label{sec:Preliminaries}}

\subsection{Thermodynamic formalism}

\textsf{For general knowledges of the Thermodynamic Formalism, a great
reference is the book written by Parry and Pollicott \cite{Parry:1990tn}.
The reparametrization method is discussed in detail in Sambarino's
work \cite{Sanders:2014wv}.}

\subsubsection{Flows and reparametrization}

\textsf{Let $X$ be a compact metric space with a continuous flow
$\phi=\{\phi_{t}\}_{t\in\real}$ on $X$ without any fixed point,
and $\mu$ is a $\phi-$invariant probablity measure on $X$. Consider
a positive continuous function $F:X\to\real_{>0}$ and define 
\[
\kappa(x,t):=\int_{0}^{t}F(\phi_{s}(x))\mbox{d}s.
\]
The function $\kappa$ satisfies the cocycle property $\kappa(x,t+s)=\kappa(x,t)+\kappa(\phi_{t}x,s)$
for all $x,t\in\real$ and $x\in X$.}

\textsf{Since $F>0$ and $X$ is compact, $F$ has a positive minimum
and $\kappa(x,\cdot)$ is an increasing homeomorphism of $\real$.
We then have a map $\alpha:X\times\real\to\real$ such that 
\[
\alpha(x,\kappa(x,t))=\kappa(x,\alpha(x,t))=t.
\]
for all $(x,t)\in X\times\real.$ }
\begin{defn}
Let $F:X\to\real$ be a positive continuous function. The \textit{reparametrization}
of the flow $\phi$ by $F$ is the flow $\phi^{F}=\{\phi_{t}^{F}\}_{t\in\real}$
defined by $\phi_{t}^{F}(x)=\phi_{\alpha_{(x,t)}}(x)$. 
\end{defn}

\begin{defn}
Two continuous functions $F,G:X\to\real$ are \textit{Livšic cohomologous}
if there exists a continuous function $V:X\to\real$ which is $C^{1}$
in the flow direction such that 
\[
F(x)-G(x)=\left.\dfrac{\partial}{\partial t}\right|_{t=0}V(\phi_{t}(x)),
\]
and we denote this relation by $F\sim G$.
\end{defn}

\subsubsection{Periods and measures}

\textsf{Let $O$ be the set of closed orbits of $\phi$. For $\tau\in O$,
let $l(\tau)$ be the period of $\tau$ with respect to $\phi$ ,
then the period of $\tau$ with respect to the reparametrized flow
$\phi^{F}$ is 
\[
\int_{0}^{l(\tau)}F(\phi_{s}(x))\mbox{d}s,
\]
where $x$ is any point on $\tau.$ Let $\delta_{\tau}$ be the Lebesgue
measure supported by the orbit $\tau$, and we denote 
\[
\langle\delta_{\tau},F\rangle=\int_{0}^{l(\tau)}F(\phi_{s}(x))\mbox{d}s.\mbox{ }
\]
}

\textsf{If $\mu$ is a $\phi-$invariant probablity measure on $X$
and $F:X\to\real$ is a continuous function, and let $\phi^{F}$ be
the reparametrization of $\phi$ by $F$. We define $\widehat{F\cdot\mu}$
to be the probablity measure: for any continuous function $G$ on
$X$ 
\[
\widehat{F.\mu}(G)=\frac{\int_{X}G\cdot F\mbox{d}\mu}{\int_{X}F\mbox{d}\mu}.
\]
Then $\widehat{F\cdot\mu}$ is a $\phi^{F}-$invariant probablity
measure.}

\subsubsection{Entropy, pressure and equilibrium states}

\textsf{We denote by $h_{\phi}(\mu)$ the }\textit{measure theoretic
entropy}\textsf{ of $\phi$ with respect to a $\phi-$invariant probablity
measure $\mu$ (cf. \cite{Parry:1990tn} for a precise definition).
Let $\mathcal{M}^{\phi}$ denote the set of $\phi-$invariant probablity
measures, and $C(X)$ denote the set of continuous functions on $X$.
The }\textit{pressure}\textsf{ of a function $F:X\to\mathbb{R}$ is
defined as 
\[
P_{\phi}(F):=\sup_{m\in\mathcal{M}^{\phi}}\left(h_{\phi}(m)+\int_{X}F\mbox{d}m\right).
\]
We define the }\textit{topological entropy}\textsf{ of the flow $\phi$
by 
\[
h_{\phi}=P_{\phi}(0).
\]
If there is no ambiguity on which flow is referred, for example $\phi$,
then we might drop the subscript $\phi$ and use $h$ to denote the
topological entropy, and $h(\mu)$ to denote the measure theoretic
entropy of $\phi$ with respect to $\mu$.}

\textsf{For a continuous function $F$, if there exists a measure
$m\in\mathcal{M}^{\phi}$ on $X$ such that 
\[
P_{\phi}(F)=h_{\phi}(m)+\int_{X}F\mbox{d}m,
\]
then $m$ is called an }\textit{equilibrium state}\textsf{ of $F$,
and denoted it by $m=m_{F}$. An equilibrium state of the function
$F\equiv0$ is called a }\textit{measure of maximum entropy}\textsf{.}

\begin{rem}
\label{rem: pressure}From the definition of the pressure, we list
two immediate properties:

\begin{itemize} 

\item[1.] ${\displaystyle h_{\phi}=\sup_{m\in\mathcal{M}^{\phi}}h(m).}$

\item[2.] $P_{\phi}$ is monotone, in the sense that if $F\geq G$
then $P_{\phi}(F)\geq P_{\phi}(G)$.

\end{itemize}
\end{rem}

\textsf{The following Abramov formula relates the measure theoretic
entropies of the flow $\phi$ and its reparametrization $\phi^{F}$.}

\begin{thm*}
[Abramov formula, \cite{Abramov:1959ww}] Suppose $\phi$ is a continuous
flow on $X$ and $\phi^{F}$ is the reparametrization of $\phi$ by
a positive continuous function $F$, then for all $\mu\in\mathcal{M}^{\phi}$
\[
h_{\phi^{F}}(\widehat{F.\mu})=\frac{h_{\phi}(\mu)}{\int_{X}F\mbox{d}\mu}.
\]

\end{thm*}

\textsf{The following Bowen's formula links the topological entropy
of the reparametrized flow $\phi^{F}$ and the reparametrization function
$F$. }
\begin{thm}
[Bowen's formula, Sambarino \cite{Sambarino:2014jv}]\label{thm: bowen's formula}If
$\phi$ is a continuous flow on a compact metric space $X$ and $F:X\to\real$
is a positive continuous function, then 
\[
P_{\phi}(-hF)=0
\]
if and only if $h=h_{\phi^{F}}$. Moreover, if $h=h_{\phi^{F}}$ and
$m$ is an equilibrium state of $-hF$, then $\widehat{F.m}$ is a
measure of maximal entropy of the reparametrized flow $\phi^{F}$
.
\end{thm}

\subsubsection{Anosov flow }

\textsf{A $C^{1+\alpha}$ flow $\phi_{t}:X\to X$ on a compact manifold
$X$ is called}\textit{ Anosov} \textsf{if there is a continuous splitting
of the unit tangent bundle $T^{1}X=E^{0}\oplus E^{s}\oplus E^{u}$,
where $E^{0}$ is the one-dimensional bundle tangent to the flow direction,
and there exists $C,\lambda>0$ such that $\left\Vert D\phi_{t}\left|E^{s}\right|\right\Vert \leq Ce^{-\lambda t}$
and $\left\Vert D\phi_{-t}\left|E^{u}\right|\right\Vert \leq Ce^{-\lambda t}$
for $t\geq0$. We say that the flow is} \textit{transitive}\textsf{
if there is a dense orbit.} 
\begin{example*}
Let $M$ be a compact Riemannian manifold with negative sectional
curvature and $\phi_{t}:T^{1}M\to T^{1}M$ is the geodesic flow on
the unit tangent bundle of $M$. Then $\phi_{t}:T^{1}M\to T^{1}M$
is a transitive Anosov flow. 
\end{example*}

\textsf{Recall that a function $F:X\to\real$ is called }\textit{$\alpha-$Hölder
continuous }\textsf{if there exists $C>0$ and $\alpha\in(0,1]$ such
that for all $x,y\in X$ we have $\left|F(x)-F(y)\right|\leq C\cdot d_{X}(x,y)^{\alpha}$.
In most cases, we will abbreviate }\textit{$\alpha-$Hölder continuous
}\textsf{to }\textit{Hölder continuous. }

\textsf{If $\phi_{t}$ is a transitive Anosov flow on a compact manifold
$X$, we know more about the pressure and equilibrium states. }
\begin{thm}
[Bowen-Ruelle \cite{Bowen:1975vu}] \label{Thm: unique eq state}If
$\phi_{t}$ is a transitive Anosov flow on a compact manifold $X$,
then for each $F:X\to\real$ Hölder continuous function, there exists
a \textbf{unique} equilibrium state $m_{F}$ of $F$, which is also
known as the \textbf{Gibbs measure} of $F$. Moreover if $F$ and
$G$ are Hölder continuous functions such that $m_{F}=m_{G}$, then
$F-G$ is Livšic cohomologous to a constant.
\end{thm}

\begin{rem*}
Because the equilibrium state $m_{F}$ of $F$ is unique, we know
that $m_{F}$ is ergodic. i.e., the Gibbs measure of $F$ is ergodic. 
\end{rem*}

\textsf{Recall that $O$ is the set of period orbits of $\phi$. For
a continuous function $F:X\to\real_{>0}$ and $T\in\real,$ we define
\[
R_{T}(F)=\{\tau\in O:\langle\delta_{\tau},F\rangle\leq T\}.
\]
}

\begin{prop}
[Bowen \cite{Bowen:1972ws}]

The topological entropy of a transitive Anosov flow $\phi$ is finite
and positive. Moreover, 
\[
h_{\phi}=\lim_{T\to\infty}\frac{1}{T}\log\#\{\tau\in O:l(\tau)\leq T\}.
\]

If $F:X\to\real$ is a positive Hölder continuous function, then 
\[
h_{F}:=h_{\phi^{F}}=\lim_{T\to\infty}\log\#R_{T}(F),
\]
is finite and positive.
\end{prop}

\begin{thm}[Equidistribution, Bowen \cite{Bowen:1972ws}, Parry-Pollicott \cite{Parry:1990tn}]
\label{Thm: Equidistribution}

Suppose $\phi$ is a transitive Anosov flow on a compact manifold
$X$. Then there exists a unique probablity measure of maximum entropy
$\mu_{\phi}$. Moreover, for all continuous function $G$ on $X$,
we know 
\[
\int_{X}G\mbox{d}\mu_{\phi}=\mu_{\phi}(G)=\lim_{T\to\infty}\frac{1}{\#R_{T}(1)}\sum_{\tau\in R_{T}(1)}\frac{\langle\delta_{\tau},G\rangle}{\langle\delta_{\tau},1\rangle}=\lim_{T\to\infty}{\displaystyle \frac{{\displaystyle \sum_{\tau\in R_{T}(1)}}\langle\delta_{\tau},G\rangle}{{\displaystyle \sum_{\tau\in R_{T}(1)}\langle\delta_{\tau},1\rangle}}}.
\]
The probablity measure $\mu_{\phi}$ is called the \textbf{\textit{Bowen-Margulis
measure}} of the flow $\phi$.

\end{thm}

\subsubsection{Livšic type theorems}

\textsf{Recall that two continuous functions $F,G:X\to\real$ are}
\textit{Livšic cohomologous}\textsf{ if there exists a continuous
function $V:X\to\real$ which is $C^{1}$ in the flow direction such
that} 
\[
F(x)-G(x)=\left.\dfrac{\partial}{\partial t}\right|_{t=0}V(\phi_{t}(x)).
\]

\begin{rem}
By definition, the following properties are immediate: \begin{itemize} 

\item[1.] If $F$ and $G$ are Livšic cohomologous then they have
the same integral over any $\phi-$invariant measure.

\item[2.] The pressure $P_{\phi}(F)$ only depends on the Livšic
cohomology class of $F$.

\item[3.]\textsf{ }$R_{T}(F)$ only depends on the Livšic cohomology
class of $F$.

\end{itemize}
\end{rem}

\textsf{In the rest of this note, we will only discuss the Livšic
cohomology of }\textit{Hölder }\textsf{continuous functions on $X$.
Specifically, two }\textit{Hölder }\textsf{continuous $F,G:X\to\real$
are called Livšic cohomologous if there exists a}\textit{ Hölder }\textsf{continuous
$V:X\to\real$ which is $C^{1}$ in the flow direction such that}
\[
F(x)-G(x)=\left.\dfrac{\partial}{\partial t}\right|_{t=0}V(\phi_{t}(x)).
\]

\begin{thm}[Livšic Theorem \cite{Livsic:2007br}]\label{Thm Liv=000161ic-Theorem}

Let $\phi_{t}:X\to X$ be a transitive Anosov flow. Let $F:X\to\real$
be a Hölder continuous function such that $\langle\delta_{\tau},F\rangle=\int_{0}^{l(\tau)}F\circ\phi_{t}(x_{\tau})dt=0$
for each $\phi-$closed orbit $\tau$ for any $x_{\tau}\in\tau$,
then $F$ is cohomologous to $0$.

\end{thm}

\begin{thm}[Positive Livšic Theorem \cite{Sambarino:2014jv}]\label{Thm:positive-Liv=000161ic-Theorem}

Let $\phi_{t}:X\to X$ be a transitive Anosov flow. Let $F:X\to\real$
be a Hölder continuous function such that $\langle\delta_{\tau},F\rangle>0$
for each $\phi-$closed orbit $\tau$, then $F$ is cohomologous to
a Hölder continuous function $G(x)$ such that $G(x)>0$, $\forall x\in X$.

\end{thm}

\textsf{For the geodesic flow on the unit tangent bundle of negative
curved manifolds, we have a similar theorem as the following:}

\begin{thm}[Nonnegative Livšic Theorem \cite{Lopes:2005hp}]\label{Thm Nonnegative-Liv=000161ic-Theorem}

Suppose $\Sigma$ is a compact Riemannian manifold with negative sectional
curvature. Let $\phi_{t}:T^{1}\Sigma\to T^{1}\Sigma$ be the geodesic
flow on $T^{1}\Sigma$. Let $F:T^{1}\Sigma\to\real$ be a Hölder continuous
function such that $\langle\delta_{\tau},F\rangle\geq0$ for each
$\phi-$closed orbit $\tau$, then $F$ is cohomologous to a Hölder
continuous function $G(x)$ such that $G(x)\geq0$, $\forall x\in T^{1}\Sigma$.

\end{thm}

\textsf{Here we recall the Anosov closing lemma \cite{Anosov:1967vf}:}

\begin{thm}[Anosov Closing Lemma]\label{Thm Anosov Closing lemma}

Let $\phi_{t}:X\to X$ be a transitive Anosov flow. Then for all $\varepsilon>0$
there exists $\delta=\delta(\varepsilon,d_{X})>0$ such that if for
$v\in X$, $t>0$ satisfying
\[
d_{X}(v,\phi_{t}v)<\delta,
\]
then there exists a closed orbit $\tau_{w}=\{\phi_{s}w\}_{s=0}^{t'}$
of period $t'$, where $\left|t-t'\right|<\varepsilon$, such that
\[
d_{X}(\phi_{s}v,\phi_{s}w)<\varepsilon\quad\mbox{for }0\leq s\leq t.
\]

\end{thm}

\subsubsection{Variance and derivatives of the pressure}

\textsf{In this subsection we shall recall the definition and basic
properties of the variance. Let $\phi_{t}:X\to X$ be a transitive
Anosov flow on a compact metric space $X$, and $C^{\alpha}(X)$ be
the set of $\alpha$-Hölder continuous function on $X$.}

\begin{defn}
Suppose $F\in C^{\alpha}(X)$ and $m_{F}$ is the equilibrium state
of $F$. For any $G\in C^{\alpha}(X)$ we define the\textit{ variance}
of $G$ with respect to $m_{F}$ by 
\[
\mathrm{Var}(G,m_{F}):=\lim_{T\to\infty}\frac{1}{T}\int_{X}\left(\int_{0}^{T}G(\phi_{t}(x))\mbox{d}t-T\int G\mbox{d}m_{F}\right)^{2}\mbox{d}m_{F}(x)
\]
 
\end{defn}

\textsf{The following properties give us some handy formulas of the
the derivatives of the pressure.}

\begin{prop}
[Parry-Pollicott, Prop 4.10, 4.11 \cite{Parry:1990tn}] Suppose that
$\phi_{t}:X\to X$ is a transitive Anosov flow on a compact metric
space $X$, and $F,G\in C^{\alpha}(X).$ If $m_{F}$ is the equilibrium
state of $F$, then \begin{itemize}

\item[1.] The function $t\mapsto P_{\phi}(F+tG)$ is analytic, 

\item[2.] The first derivative is given by 
\[
\left.\frac{dP_{\phi}(F+tG)}{dt}\right|_{t=0}=\int_{X}G\mbox{d}m_{F},
\]

\item[3.] If $\int G\mbox{d}m_{F}=0$, then the second derivative
could be formulated as 
\[
\left.\frac{d^{2}P_{\phi}(F+tG)}{dt^{2}}\right|_{t=0}=\mathrm{Var}(G,m_{F}),
\]

\item[4.] If $\mathrm{Var}(G,m_{F})=0$, then $G$ is Livšic cohomologous
to zero. 

\end{itemize}
\end{prop}

\subsubsection{Pressure metric}

\textsf{Here we keep the same setting as in previous subsection that
$\phi_{t}:X\to X$ is a transitive Anosov flow on a compact metric
space $X$. We consider the space $\mathcal{P}(X)$ of Livšic cohomology
classes of pressure zero Hölder continuous functions on $X$, i.e.
\[
\mathcal{P}(X):=\left\{ F;F\in C^{\alpha}(X)\mbox{ for some \ensuremath{\alpha\:}and \ensuremath{P_{\phi}(F)=0}}\right\} /\sim.
\]
The tangent space of $\mathcal{P}(X)$ at $F$ is 
\[
T_{F}\mathcal{P}(X)=\mathrm{ker}\left((\mathbf{D}P_{\phi})(F)\right)=\left\{ G;\mbox{ \ensuremath{G\in C^{\alpha}(X)\mbox{ for some \ensuremath{\alpha\:}and}\int Gdm_{F}=0}}\right\} /\sim,
\]
where $m_{F}$ is the equilibrium state of $F$. }

\textsf{Since the variance vanishes only on functions $G$ that are
cohomologous to 0, 
\[
\left\Vert G\right\Vert _{P}^{2}:=\frac{\mathrm{Var}(G,m_{F})}{-\int F\mbox{d}m_{F}}
\]
 is positive definite on equivalence classes of functions cohomologous
to 0, and thus a metric on $T_{F}\mathcal{P}(X)$. We call this metric
$||\cdot||_{P}$ the }\textit{pressure metric}\textsf{ on $T_{F}\mathcal{P}(X)$. }

\begin{prop}
\label{prop: pressure metric}If $\{c_{t}\}_{t\in(-1,1)}$ is a smooth
one parameter family contained in $\mathcal{P}(X)$, then 
\[
\left\Vert \dot{c}_{0}\right\Vert _{P}^{2}=\frac{\int\ddot{c}_{0}\mbox{d}m_{c_{0}}}{\int c_{0}\mbox{d}m_{c_{0}}},
\]
where $\dot{c}_{0}=\left.\frac{d}{dt}c_{t}\right|_{t=0}$ and $\ddot{c}_{0}=\left.\frac{d^{2}}{dt^{2}}c_{t}\right|_{t=0}$.\end{prop}
\begin{proof}
This follows the direction computation of the (Gâteaux) second derivative
of $P_{\phi}(c_{t})$: 
\begin{align*}
\left.\frac{d^{2}}{dt^{2}}P_{\phi}(c_{t})\right|_{t=0} & =(\mathbf{D}^{2}P_{\phi})(c_{0})(\dot{c}_{0},\dot{c}_{0})+(\mathbf{D}P_{\phi})(c_{0})(\ddot{c}_{0})\\
 & =\mathrm{Var}(\dot{c}_{0},m_{c_{0}})+\int\ddot{c}_{0}\mbox{d}m_{c_{0}}.
\end{align*}

Since $P_{\phi}(c_{t})=0$, we have 
\[
\left\Vert \dot{c}_{0}\right\Vert _{P}^{2}=\frac{\mathrm{Var}(\dot{c}_{0},m_{c_{0}})}{-\int c_{0}\mbox{d}m_{c_{0}}}=\frac{\int\ddot{c}_{0}\mbox{d}m_{c_{0}}}{\int c_{0}\mbox{d}m_{c_{0}}}.
\]

\end{proof}

\subsection{Negatively curved manifolds and the group of isometries}

\textsf{In this subsection, we survey several facts of $\delta-$hyperbolic
spaces and their group of isometries. A good reference is the book
\cite{Ghy:1990} edited by Ghys and de la Harpe.}

\subsubsection{$\delta$--hyperbolic spaces}

\textsf{A metric space $(X,d)$ is said to be}\textit{ geodesic}\textsf{
if any two points $x,y\in X$ can be joined be a}\textit{ geodesic
segment}\textsf{ $[x,y]$ that is a naturally parametrized path from
$x$ to $y$ whose length is equal to $d(x,y)$, and is called }\textit{proper}\textsf{
if all closed balls are compact. }
\begin{defn}
A geodesic metric space $(X,d)$ is called $\delta-$\textit{hyperbolic}
(where $\delta\geq0$ is some real number) if for any geodesic triangle
in $X$ each side of the triangle is contained in the $\delta-$neighborhood
of the union of two other sides. A metric space $(X,d)$ is called
\textit{hyperbolic} if it is $\delta-$\textit{hyperbolic }for some
$\delta\geq0$.

\textsf{In the following, we lists two types of hyperbolic spaces
appearing in this note. }
\end{defn}

\begin{example*}
:\begin{itemize} 

\item[1.]\textit{Pinched Hadamard manifold}\textsl{ $(\widetilde{M},d_{\widetilde{g}})$:}
a complete and simply connected Riemannian manifold $(\widetilde{M},\widetilde{g})$
whose sectional curvature is bounded by two negative numbers. The
metric on $\widetilde{M}$ is the distance function $d_{\widetilde{g}}$
induced by the Riemannian metric $\widetilde{g}$. 

\item[2.]The \textit{Cayley graph} $C(G,S)$ and its \textit{word
metric} $w$: given a finitely generated group $G$ and a finite generating
set $S$ of $G$, $C(G,S)$ is a graph whose vertices are elements
of $G$. Two vertices $g,h\in G$ are connected by an edge if and
only if there is a generator $s\in S$ such that $h=gs$. The word
metric $w$ on $C(G,S)$ is defined by assuming that each edge has
unit length, and $w(g,h)$ is the minimum of the length of all paths
connecting $g$ and $h$.

\end{itemize}
\end{example*}

\begin{rem}
A group $G$ is called \textit{hyperbolic} if for one (and for all)
finite generating set the Cayley graph is hyperbolic. For example,
finitely generated free groups and surface groups for surfaces with
genus $>1$.
\end{rem}

\textsf{We say that two geodesics ray $\tau_{1}:[0,\infty)\to X$
and $\tau_{2}:[0,\infty)\to X$ are }\textit{equivalent} \textsf{and
write $\tau_{1}\sim\tau_{2}$ if there is a $K>0$ such that for all
$t>0$
\[
d(\tau_{1}(t),\tau_{2}(t))<K.
\]
It is easy to see that $\sim$ is indeed an equivalence relation on
the set of geodesic rays. We then define the }\textit{geometric boundary}\textsf{
$\vbdy X$ of $X$ by 
\[
\vbdy X:=\{[\tau]:\mbox{\ensuremath{\tau}\ is a geodesic ray in }X\}.
\]
Moreover, we know that when $X$ is proper, $\vbdy X$ is metrizable.
More precisely, the following }\textit{visual metric,}\textsf{ given
by Gromov, defines a metric on $\vbdy X$.}

\begin{defn}
Let $(X,d)$ be a $\delta$--hyperbolic proper metric space. Let $a>1$
and let $o\in X$ be a basepoint. We say that a metric $d_{a}$ on
$\vbdy X$ is a \textit{visual metric} with respect to the basepoint
$o$ and the visual parameter $a$ if there exists $C>0$ such that
for any two distinct point $\xi,\eta\in\vbdy X$, and for any biinfinite
geodesic $\tau$ connecting $\xi$ to $\eta$ in $X$ we have 
\[
\frac{1}{C}a^{-d(o,\tau)}\leq d_{a}(\xi,\eta)\leq Ca^{-d(o,\tau)}.
\]

\end{defn}

\begin{thm*}
[Gromov, cf. Theorem 1.5.2 \cite{Bourdon:1995uy}] There is $a_{0}>1$
such that for any basepoint $o\in X$ and any $a\in(1,a_{0})$ the
boundary $\vbdy X$ admits a visual metric $d_{a}$ with respect to
$o$.
\end{thm*}

\begin{rem}
[ cf. Ch.11 \cite{Coornaert:1990tj}, Remark 1.5.3 \cite{Bourdon:1995uy}]:\begin{itemize}

\item[1.] For a pinched Hadamard manifold $M$ whose sectional curvature
is not larger than $-b^{2}$, then the boundary $\vbdy M$ admits
a visual metric $d_{a}$ for all $a\in(1,e^{b}]$. 

\item[2.] Let $d_{a}$ and $d_{a'}$ be two different visual metrics
with respect to a fixed basepoint $o\in X$ and different visual parameters
$a$ and $a'$, then $d_{a}$ and $d_{a'}$ are Hölder equivalent.
i.e., there exists $C\geq1$ and $\alpha=\frac{\log a'}{\log a}$
such that for $\xi,\eta\in\vbdy X$ 
\[
\frac{1}{C}\cdot\left(d_{a}(\xi,\eta)\right)^{\alpha}\leq d_{a'}(\xi,\eta)\leq C\cdot\left(d_{a}(\xi,\eta)\right)^{\alpha}.
\]

\item[3.] Let $d_{a}$ and $d'_{a}$ be two different visual metrics
with respect to a fixed visual parameter $a$ and different basepoints
$o,o'\in X$, then the metric $d_{a}$ and $d'_{a}$ are Lipchitz.
i.e., there exists $C\geq1$ such that for all $\xi,\eta\in\vbdy X$
\[
\frac{1}{C}\cdot d_{a}(\xi,\eta)\leq d'_{a}(\xi,\eta)\leq C\cdot d_{a}(\xi,\eta).
\]

\end{itemize}
\end{rem}

\subsubsection{Quasi-isometry}

\begin{defn}
A function $q:X\to Y$ from a metric space $(X,d_{X})$ to a metric
space $(Y,d_{y})$ is called a \textit{$(C,L)$-quasi-isometry embedding}
if there is $C,L>0$ such that: 

For any $x,x'\in X$, we have 
\[
\frac{1}{C}d_{X}(x,x')-L\leq d_{Y}(q(x),q(x'))\leq C\cdot d_{X}(x,x')+L.
\]

If, in addition, there exists an \textit{approximate inverse} map
$\bar{q}:Y\to X$ that is a $(C,L)$-quasi-isometric embedding such
that for all $x\in X$ and $y\in Y$
\[
d_{X}(\bar{q}q(x),x)\leq L,\qquad d_{Y}(q\bar{q}(y),y)\leq L,
\]
then we call $q$ a \textit{$(C,L)$-quasi-isometry}. $(X,d_{X})$
and $(Y,d_{Y})$ are called \textit{quasi-isometric.}
\end{defn}

\textsf{In most cases, the quasi-isometry constants $C$ and $L$
do not matter, so we shall use the words quasi-isometry and quasi-isometry
embedding without specifying constants. }

\begin{thm}
[Bourdon, Theorem 1.6.4 \cite{Bourdon:1995uy}]\label{thm: Bourdon}Let
$(X,d_{X})$ and $(Y,d_{Y})$ be hyperbolic spaces. Suppose the boundaries
equip with visual metrics. Then 

\begin{itemize}

\item[1.]Any quasi-isometry embedding $q:X\to X'$ extends to a bi-Hölder
embedding $q:\vbdy X\to\vbdy Y$ with respect to the corresponding
visual metrics.

\item[2.]Any quasi-isometry $q:X\to X'$ extends to a bi-Hölder homemorphism
$q:\vbdy X\to\vbdy Y$ with respect to the corresponding visual metrics. 

\end{itemize}
\end{thm}

\begin{defn}
A $(C,L)-$\textit{quasi-geodesic} is a $(C,L)-$quasi-isometry embedding
$q:\mathbb{R}\to X$.
\end{defn}

\begin{thm}
[Morse Lemma, cf. Ch.5, Theorem 6 \cite{Ghy:1990} ] \label{thm: Morse Lemma}Suppose
$X$ and $Y$ are hyperbolic spaces , and $q:X\to Y$ is a $(C,L)$-quasi-isometry.
Then every geodesic $\gamma\subset X$ its image $q(\gamma)$ is a
quasi-geodesic on $Y$ and is within a bounded distance $R$ from
a geodesic on $Y$. Moreover, this constant $R$ is only depending
on $X$, $Y$ and the quasi-isometry constants $C$ and $L.$ 
\end{thm}

\begin{rem}
When the space $Y$ is a pinched Hadamard manifold, we have a stronger
result of the above theorem. Specificity, every geodesic $\gamma\subset X$
its image $q(\gamma)$ is a quasi-geodesic on $Y$ and is within a
bounded distance $R$ from a \textit{unique} geodesic on $Y$.
\end{rem}

\textsf{Let $X$ be a hyperbolic space, we denote its group of isometries
by $\mathrm{Isom}(X)$. The following lemma connects some subgroups
of $\mathrm{Isom}(X)$ and the hyperbolic space $X$. }

\begin{thm}
[Švarc-Milnor lemma, cf. Lemma 3.37 \cite{Kapovich:2009fk} ]\label{thm:Svarc-milnor}
Let $X$ be a proper geodesic metric space. Let $G$ be a subgroup
of $\mathrm{Isom}(X)$ acting properly discontinuously and cocompactly
on $X$. Pick a point $o\in X$. Then the group $G$ is finitely generated;
for some choice of finitely generating set $S$ of $G$, the map $q:G\to X$,
given by $q(\gamma)=\gamma(o)$, is a quasi-isometry. Here $G$ is
given the word metric induced from $C(G,S)$.
\end{thm}

\subsubsection{Negatively curved manifolds and the group of isometries}

\textsf{Let $(X,g$) be a negatively curved compact Riemannian manifold.
The universal covering $(\widetilde{X},\widetilde{g})$ of $(X,g)$
is a pinched Hadamard manifold, and $\pi_{1}X$ is finitely generated
and acting canonically on $\widetilde{X}$. Let $\Gamma$ denote the
group of deck transformations of the covering $\widetilde{X}$. We
know that $\Gamma\subset\mathrm{Isom}(\widetilde{X})$, $\Gamma\cong\pi_{1}X$,
and $X$ is isometric to $\Gamma\backslash\widetilde{X}$. More precisely,
using generators there exists a nature isomorphism $i_{X}:\pi_{1}X\to\Gamma$,
given by $\gamma_{X}:=i_{X}(\gamma),$ $\forall\gamma\in\pi_{1}X$.
Thus, using this isomorphism we can define a $\pi_{1}X$-action on
$\widetilde{X}$ by $\gamma\cdot x=(\gamma_{X})(x)$. It's clear that
this $\pi_{1}X$-action is nothing different from the $\Gamma$-action
on $\widetilde{X}$. }

\textsf{Because $(X,g$) is negatively curved, every $\gamma\in\Gamma$
corresponds to a unique geodesic $\tau_{\gamma}^{X}$ on $X$. Besides,
each conjugacy class $[\gamma]\in[\Gamma]$ corresponds to a unique
closed geodesic geodesic $\tau_{\gamma}^{X}$ on $X$ and vice versa.
Moreover, the length of the closed geodesic $\tau_{\gamma}^{X}$ is
exactly the translation distance of $\gamma\in\pi_{1}X$. i.e. $l_{g}(\tau_{\gamma}^{X})=l_{g}(\gamma):=d_{g}(x,\gamma\cdot x)=d_{g}(x,\gamma_{X}(x))$.}

\begin{thmdefn}[Margulis \cite{Margulis:1970gj}] Let $(X,g)$ be
a compact negatively curved Riemannian manifold and $\Gamma$ be the
group of deck transformations of $\widetilde{X}$, then the topological
entropy $h(X)$ geodesic flow on $T^{1}X$ is given by 
\[
h(X)=\lim_{T\to\infty}\frac{1}{T}\log\#\left\{ [\gamma]\in[\pi_{1}X];l_{g}(\gamma)\leq T\right\} .
\]

\end{thmdefn}

\textsf{Now let's consider a compact 3--manifold $M$ equipped with
a hyperbolic metric $h$. Then there exists a discrete and faithful
representation $\rho:\pi_{1}M\to\mbox{\ensuremath{\mathrm{Isom}}}(\hy^{3})$
such that $M\cong\rho(\pi_{1}M)\backslash\hy^{3}$ where $(\hy^{3},\widetilde{h})$
is the universal covering of $(M,h)$. For the sake of brevity, in
what follows we will denote the lifted metric of $\widetilde{h}$
on $\hy^{3}$ by $h.$ }

\begin{defn}
Let $\Gamma$ be a discrete subgroup of $\mathrm{Isom}(\hy^{3})$,
the \textit{limit set} $\Lambda(\Gamma)$\textsf{ }is the set of limit
points $\Gamma x$ for any $x\in\hy^{3}$. 
\end{defn}

\begin{defn}
The \textit{critical exponent} \textsf{$\delta_{\Gamma}$} is defined
as following: 
\[
\delta_{\Gamma}:=\inf\{s;\sum_{\gamma\in\Gamma}e^{-sd_{h}(x,\gamma x)}<\infty\},
\]
for any point $x\in\hy^{3}$ and $d_{h}$ is the hyperbolic distance
on $\hy^{3}$.
\end{defn}

\begin{defn}
A discrete subgroup $\Gamma$ of $\Isom(\hy^{3})$ is called \textit{convex
cocompact} if $\Gamma$ acts cocompactly on the convex hull of the
limit set of $\Gamma$. i.e., $\Gamma\backslash\mbox{Conv}(\Lambda(\Gamma))$
is compact .
\end{defn}

\begin{thm}
[Sullivan \cite{Sullivan:1979vs}] Suppose $\Gamma$ is a non-elementary,
convex cocompact, and discrete subgroup of $\Isom(\hy^{3})$, then
\[
\delta_{\Gamma}=\lim_{T\to\infty}\frac{1}{T}\log\#\{[\gamma]\in[\Gamma];l_{h}(\gamma)\leq T\},
\]
 where $l_{h}(\gamma)=d_{h}(o,\gamma o)$, $o$ is the origin of $\hy^{3}$. 
\end{thm}

\subsection{Hölder cocycles}

\textsf{Let $(X,g)$ be a compact negatively curved manifold, $\widetilde{X}$
be its universal covering, and $\Gamma$ be the group of deck transformations
of the covering $\widetilde{X}$. Recall that the $\pi_{1}X$-action
on $\widetilde{X}$ is defined by $\gamma\cdot x=i_{X}(\gamma)(x)$,
where $i$ is the isomorphism $i_{X}:\pi_{1}\Sigma\to\Gamma$. }
\begin{defn}
A\textit{ Hölder cocycle} is a function $c:\pi_{1}X\times\vbdy\widetilde{X}\to\real$
such that 
\[
c(\gamma_{0}\gamma_{1},x)=c(\gamma_{0},\gamma_{1}\cdot x)+c(\gamma_{1},x)
\]
for any $\gamma_{0},\gamma_{1}\in\pi_{1}X$ and $x\in\vbdy\widetilde{X}$,
and $c(\gamma,\cdot)$ is Hölder continuous for every $\gamma\in\pi_{1}X$.
\end{defn}

\textsf{Given a Hölder cocycle $c$ we define the }\textit{periods}
\textsf{of $c$ to be the number 
\[
l_{c}(\gamma):=c(\gamma,\gamma_{X}^{+})
\]
where $\gamma_{X}^{+}$ is the attracting fixed point of $\gamma$$\in\pi_{1}X\backslash\{e\}$
on $\vbdy\widetilde{X}$. }

\begin{rem}
The cocycle property implies that the period of an element $\gamma$
only depends on its conjugacy class $[\gamma]\in[\pi_{1}X]$. 
\end{rem}

\textsf{Two cocycles $c$ and $c'$ are said to be cohomologous if
there exists a Hölder continuous function $U:\vbdy\widetilde{X}\to\real$
such that for all $\gamma\in\pi_{1}X$ one has 
\[
c(\gamma,x)-c'(\gamma,x)=U(\gamma\cdot x)-U(x).
\]
One easily deduces from the definition that the set of periods of
a cocycle is a cohomological invariant.}

\begin{defn}
The\textit{ exponential growth} \textit{rate} for a Hölder cocycle
$c$ is defined as:
\[
h_{c}:=\limsup_{T\to\infty}\frac{1}{T}\log\#\{[\gamma]\in[\pi_{1}X]:l_{c}(\gamma)\leq T\}.
\]

\end{defn}

\subsubsection{From cocycle cohomology to Livšic cohomology}

\begin{thm}[Ledrappier \cite{Ledrappier:1994uy}]\label{Thm: Ledrappier}

For each Hölder cocycle $c:\pi_{1}X\times\vbdy\widetilde{X}\to\real$,
there exists a Hölder continuous function $F_{c}:T^{1}X\to\real$,
such that for all $\gamma\in\pi_{1}X-\{e\}$, one has 
\[
l_{c}(\gamma)=\int_{[\gamma]}F_{c}.
\]
 The map $c\mapsto F_{c}$ induces a bijection between the set of
cohomology classes of $\real-$valued Hölder cocycles, and the set
of Livšic cohomology classes of Hölder continuous functions from $T^{1}X$
to $\real$.

\end{thm}

\textsf{Using the above Theorem \ref{Thm: Ledrappier}, Sambarino
give the following reparametrization theorem in \cite{Sambarino:2014jv}. }

\begin{thm}[Sambarino \cite{Sambarino:2014jv}]\label{Thm: Sambarino}

Let $c$ be a Hölder cocycle with positive periods such that $h_{c}$
is finite. Then the action of $\Gamma$ on $\left(\vbdy\widetilde{X}\times\vbdy\widetilde{X}-\mbox{diagonal}\right)\times\real$
via $c$, that is, 
\[
\gamma(x,y,s)=(\gamma x,\gamma y,s-c(\gamma,y)),
\]
is proper and compact. Moreover, the flow $\psi$ on $\pi_{1}X\backslash(\left(\vbdy\widetilde{X}\times\vbdy\widetilde{X}-\mbox{diagonal}\right)\times\real)$,
defined by 
\[
\psi_{t}\Gamma(x,y,s)=\Gamma(x,y,s-t),
\]
is conjugated to $\phi^{F_{c}}:T^{1}X\to T^{1}X$ which is the reparametrization
of the geodesic flow $\phi$ on $T^{1}X$ by a Hölder continuous function
$F_{c}$ provided $l_{c}(\gamma)=\int_{[\gamma]}F_{c}$ for all $[\gamma]\in[\pi_{1}X]$
. Furthermore, the conjugating map is also Hölder continuous, and
the topological entropy of $\psi$ is $h_{c}$.

\end{thm}

\subsection{Immersed surfaces in hyperbolic 3--manifolds}

\textsf{In this subsection, we review some well-known facts about
immersed surfaces in hyperbolic 3--manifolds. Let $\Sigma$ be a differentiable
2-manifold and $M$ be a 3-manifold, we say a differentiable mapping
$f:\Sigma\to M$ is an }\textit{immersion}\textsf{ if $df_{p}:T_{p}\Sigma\to T_{f(p)}M$
is injective for all $p\in S$. If, in addition, $f$ is a homemorphism
onto $f(\Sigma)\subset M$, where $f(\Sigma)$ has the subspace topology
induced from $M$, we say that $f$ is an }\textit{embedding}\textsf{.
Moreover, if the induce homomorphism $f_{*}:\pi_{1}\Sigma\to\pi_{1}M$
is injective, then we call $f$ is a }\textit{$\pi_{1}$--injective}\textsf{. }

\textsf{Throughout, we consider that $M$ is a hyperbolic 3--manifold
equipped with a hyperbolic metric $h$ and $\Sigma$ is a compact,
2--dimension manifolds with negative Euler characteristic. Before
moving further, we recall several terminologies in differential geometry.
Given an immersion $f:\Sigma\to M$, let $g=f^{*}g$ be the induced
Riemannian metric on $\Sigma$, $\nabla$ be the Levi-Civita connection
on $(M,h)$, $N$ be the unit outward normal vector field to the surface
$f(\Sigma)\subset M$, and $\partial_{1}$ and $\partial_{2}$ be
the coordinate fields of $T\Sigma$. }

\textsf{The }\textit{second fundamental form}\textsf{ $B:T\Sigma\times T\Sigma\to\real$
of $f$ is the symmetric 2-tensor on $\Sigma$ defined by, locally,
\[
B(\partial_{i},\partial_{j})=\langle\partial_{i},-\nabla_{\partial_{j}}N\rangle_{h},
\]
where $\langle,\rangle_{h}$ is the hyperbolic metric $h$ on $M$. }

\textsf{The}\textit{ shape operator}\textsf{ $S_{g}:T\Sigma\to T\Sigma$
is the symmetric self-adjoint endomorphism defined by raising one
index of the second fundamental form $B$ with respect to the metric
$g.$} 

\textsf{The}\textit{ mean curvature}\textsf{ $H$ of the immersion
$f:\Sigma\to M$ (or, of the immersed surface $(\Sigma,g)$) is the
trace of the shape operator. We call an immersion $f:\Sigma\to M$
}\textit{minimal}\textsf{ if the mean curvature $H$ vanishes identically. }

\textsf{Moreover, we can relate the induced Riemannian metric $g$
and shape operator $S_{g}$ by the famous Gauss-Codazzi equations:
\begin{alignat}{2}
K_{g} & =-1+\det S_{g}, & (\mbox{Gauss eq.})\label{eq: Gauss}\\
\nabla_{df(X)}(S_{g}(Y)) & -\nabla_{df(Y)}S_{g}(X)=S_{g}([X,Y]). & \hspace{1cm}(\mbox{Codazzi eq.})\label{eq: Codazzi}
\end{alignat}
where $X,Y\in T\Sigma$ and $[\cdot,\cdot]$ is the Lie bracket on
$T\Sigma.$}
\begin{rem}
\label{rem: Gauss-Codazzi +miminal}:\begin{itemize}

\item[1.]We call real eigenvalues $\lambda_{1}$ and $\lambda_{2}$
of $S_{g}$ the \textit{principal curvatures }of the immersion $f:\Sigma\to M$.

\item[2.]The shape operator $S_{g}$ and the second fundamental form
$B$ are linked by 
\[
B(X,Y)=\langle X,S_{g}(Y)\rangle_{g},\mbox{ }\forall X,Y\in T\Sigma.
\]

\item[3.]If $f$ is a minimal immersion the Gauss-Codazzi equations
could be expressed in terms of $B$ by 
\begin{alignat*}{1}
K_{g} & =-1-\frac{1}{2}\left\Vert B\right\Vert _{g}^{2},\\
(\nabla_{\partial_{i}}B)_{jk} & =(\nabla_{\partial_{j}}B)_{ik},
\end{alignat*}
where $||\cdot||_{g}$ is the tensor norm w.r.t. metric $g$ and $\partial_{1}$
and $\partial_{2}$ are coordinate fields of $TM$. Moreover, in this
case the Gauss equation implies $K_{g}\leq-1$. i.e., $(\Sigma,g)$
is a negatively curved surface. 

\end{itemize}
\end{rem}

\subsubsection{Minimal hyperbolic germs}

\textsf{In the context of minimal hyperbolic germs, the surface $\Sigma$
is always assumed to be closed. Let $(g,B)$ be a pair consisting
of a Riemannian metric $g$ and a symmetric 2-tensor $B$ on $\Sigma$. }

\begin{defn}
A pair $(g,B)$ is called a \textit{minimal hyperbolic germ} if it
satisfies the following equations 
\[
\left\{ \begin{array}{c}
K_{g}=-1-\frac{1}{2}\left\Vert B\right\Vert _{g}^{2},\\
(\nabla_{\partial_{i}}B)_{jk}=(\nabla_{\partial_{j}}B)_{ik},\\
H=0.
\end{array}\right.
\]

\end{defn}

\textsf{Recall that $\mbox{Diff}_{0}(\Sigma)$ is the space of orientation
preserving diffeomorphisms of $\Sigma$ isotopic to the identity.
There is a natural $\mathrm{Diff}_{0}(\Sigma)$ action (i.e. by pullback)
on the space of minimal hyperbolic germs, and we are mostly interested
in the following quotient space. }

\begin{defn}
The space $\mathcal{H}$ of minimal hyperbolic germs is the quotient:
\[
\mathcal{H}=\{\mbox{minimal hyperbolic germs}\}/\mbox{Diff}_{0}(\Sigma).
\]

\end{defn}

\textsf{Taubes shows that $\mathcal{H}$ is a smooth manifold of dimension
$12g-12$ where $g$ is the genus of $\Sigma.$ The fundamental theorem
of surface theory ensures that each $(g,B)\in\mathcal{H}$ can be
integrated to an immersed minimal surface in a hyperbolic 3-manifold
with the Riemannian metric $g$ and the second fundamental form $B.$ }

\textsf{Moreover, $\mathcal{H}$ is closely related with the Teichmüller
space. To be more precise, we first recall facts in the Teichmüller
theory. The }\textit{Teichmüller space}\textsf{ $\mathcal{T}$ of
$\Sigma$ is the space of isotopic classes of complex structures on
$\Sigma$. Alternatively, by the uniformization theorem, $\mathcal{T}$
can also be identified with the space of isotopic classes of conformal
structures on $\Sigma$, i.e. conformal classes of Riemannian metrics
with curvature $-1$. It is clear that we can identify $\mathcal{T}$
with a subspace $\mathcal{F}$ of $\mathcal{H}$. Namely, the }\textit{Fuchsian
space}\textsf{: 
\[
\mathcal{F}=\{(m,0)\in\mathcal{H};\mbox{}m\mbox{ is a Reimannian metric of constant curvature }-1\}.
\]
 }

\textsf{Let $[g]$ be the conformal class of a Riemannian metric $g$
on $S$ and $X=(S,[g])$ be the Riemann surface associated with $g$.
It is well-known that $T_{X}^{*}\mathcal{T}$ the fiber of the holomorphic
cotangent bundle over $X$ can be identified with $Q(X)$ the space
of holomorphic quadratic differentials on $X$. }

\textsf{The following theorem of Hopf \cite{Hopf:1951gh} helps us
see the relation between $\mathcal{H}$ and $Q(X).$}

\begin{thm}
[Hopf \cite{Hopf:1951gh}]

If $(g,B)\in\mathcal{H}$, then $B$ is the real part of a (unique)
holomorphic quadratic differential $\alpha\in Q(X)$. More precisely,
if $(x_{1},x_{2})=x_{1}+ix_{2}=z$ is a local isothermal coordinate
of $X$ and $B=B_{11}dx_{1}^{2}+B_{22}^{2}dx_{2}^{2}+2B_{12}dx_{1}dx_{2}$,
then 
\[
\alpha(g,B)=\left(B_{11}-iB_{12}\right)(x_{1},x_{2})dz^{2}.
\]

\end{thm}

\begin{rem*}
In fact $B_{11}=-B_{22}$ because $(\Sigma,g)$ is minimal, and it
is not hard to see $\left\Vert \alpha\right\Vert _{g}=\left\Vert B\right\Vert _{g}$.
\end{rem*}

\textsf{\textit{Moreover, the space $\mathcal{H}$ admits a smooth
map to $\;T^{*}\mathcal{T}$ given by 
\begin{alignat*}{1}
\Psi:\mathcal{H} & \to\;\;\;\;T^{*}\mathcal{T}\\
(g,B) & \mapsto([g],\alpha(g,B)).
\end{alignat*}
}}

\textsf{For any two holomorphic quadratic differentials $\alpha$
and $\beta$ in $Q(X)$, the }\textit{Weil-Petersson pairing }\textsf{is
given by 
\[
\langle\alpha,\beta\rangle_{WP}=\int_{\Sigma}\frac{\alpha\overline{\beta}}{m},
\]
 where $m$ is the hyperbolic metric on $\Sigma$ conformal to $g$.
It's also well-known that this pairing defines a Kähler metric, }\textit{the
Weil-Petersson metric}\textsf{, on the Teichmüller space whose geometry
has been intensely studied. In the last section, we will discuss serval
applications of our results related with the Weil-Petersson metric
on $\mathcal{F}$. }

\textsf{We now change gear to the Gauss equation. Since every Riemannian
metric $g$ on $\Sigma$ is conformal to a unique hyperbolic metric
$m$, we can write $g=e^{2u}m$ where $e^{2u}$ is the conformal factor.
Therefore, we can rewrite the Gauss equation as the following. }

\begin{thm}
[Gauss equation, Theorem 4.2 \cite{Uhlenbeck:1983wl}]\label{thm: gauss equation}

The Gauss equation may be written, in terms of $m$, 
\[
-1-\frac{1}{2}\left\Vert B\right\Vert _{m}^{2}=K_{g}=e^{-2u}(-\Delta_{m}u-1),
\]
where $K_{g}$ is the Gaussian curvature of $(\Sigma,g)$.
\end{thm}

\textsf{From another point of view, using Uhlenbeck's result in \cite{Uhlenbeck:1983wl}
we can relate the space of minimal hyperbolic germs $\mathcal{H}$
with the character variety $\mathcal{R}(\pi_{1}(\Sigma),\psl(2,\co))$,
where $\mathcal{R}(\pi_{1}(\Sigma),\psl(2,\co))$ is the space of
conjugacy classes of representations of $\pi_{1}(S)$ into $\psl(2,\co)$.
More precisely, Uhlenbeck \cite{Uhlenbeck:1983wl} proves that for
each data $(g,B)\in\mathcal{H}$ there exists a representation $\rho:\pi_{1}(\Sigma)\to\mbox{Isom }(\hy{}^{3})\cong\psl(2,\co)$
leaving this minimal immersion invariant. i.e., there is a map 
\begin{equation}
\Phi:\mathcal{H}\to\mathcal{R}(\pi_{1}(\Sigma),\psl(2,\co)).\label{eq: charater variality-1}
\end{equation}
}

\subsubsection{Almost-Fuchsian germs}

\begin{defn}
The space of \textit{almost-Fuchsian} germs $\mathcal{AF}$ is defined
by 
\[
\mathcal{AF}=\{(g,B)\in H;\mbox{}\left\Vert B\right\Vert _{g}^{2}<2\}.
\]
\textsf{ }
\end{defn}

\begin{rem}
$\left\Vert B\right\Vert _{g}^{2}<2$ is equivalent to that the principal
curvatures $\lambda_{1}$ and $\lambda_{2}$ are between $-1$ and
$1$.
\end{rem}

\textsf{This definition is motivated by the observation of Uhlenbeck
\cite{Uhlenbeck:1983wl} that when $(g,B)\in\mathcal{AF}$ there exists
a unique quasi-Fuchsian manifold $M$, up to isometry, such that $\Sigma$
is an embedded minimal surface in $M$ with the induced metric $g$
and the second fundamental form $B$. In other words, almost-Fuchsian
germs are more than quasi-Fuchsian, but not Fuchsian. }

\textsf{In the following, we discuss a ray in $\mathcal{AF}$, which
could be considered as a path in the Teichmüller space $\mathcal{T}$.
Specifically, for a hyperbolic metric $m\in\mathcal{F}$, and a holomorphic
quadratic differential $\alpha\in Q((\Sigma,[m]))$, we consider the
ray 
\[
r(t)=(g_{t},tB)\subset\mathcal{AF},
\]
where $g_{t}$ and $B=\mathrm{Re}(\alpha)$ satisfying }$\left\Vert tB\right\Vert _{g_{t}}^{2}<2$
\textsf{. Notice that $g_{t}$ is conformal to the hyperbolic metric
$m$, so we can write $g_{t}=e^{2u_{t}}m$ where the conformal factor
$e^{2u_{t}}$ is a $C^{2}$ function on $\Sigma$. By studying the
Gauss equation, Uhlenbeck proved $u_{t}$ is smooth on $t$, hence
$r(t)$ is smooth when $t$ is small. We state that result in the
below.}

\begin{thm}
[Uhlenbeck, Theroem 4.4 \cite{Uhlenbeck:1983wl}]\label{Thm: Uhlenbeck solution curve}Consider
the maps $F:W^{2,2}(\Sigma)\times[0,\infty)\to L^{2}(\Sigma)$, 
\[
F(u,t)=\Delta_{m}u+1-e^{2u}-\frac{1}{2}\left\Vert tB\right\Vert _{m}^{2}e^{-2u},
\]
where $W^{2,2}(\Sigma)$ is the classical Sobolev space. Then there
exists $\tau_{0}>0$ and a smooth solution curve 
\begin{alignat*}{1}
c:[0,\tau_{0}] & \to W^{2,2}(\Sigma)\times[0,\infty)\\
t & \mapsto\quad(u(t),t)
\end{alignat*}
such that $c(0)=(0,0)$ and $F(c(t))=0$ for all $t\in[0,\tau_{0}]$.
\end{thm}

\section{Proof of the main result\label{sec:Proof-of-the}}

\textsf{Throughout this section, $\Sigma$ denotes a compact 2--dimensional
manifold with negative Euler characteristic. Let $f:\Sigma\to M$
be a $\pi_{1}-$injective immersion from $\Sigma$ to a hyperbolic
3-manifold $M$ and $\Gamma$ be the copy of $\pi_{1}\Sigma$ in $\mathrm{Isom}(\hy^{3})$
induced by the immersion $f$}. \textsf{More precisely, let $\rho:\pi_{1}M\to\mbox{\ensuremath{\mathrm{Isom}}}(\hy^{3})$
be the discrete and faithful representation, up to conjugacy, corresponding
to $M$, i.e. $M=\rho(\pi_{1}M)\backslash\hy^{3}$. Then $\Gamma=\rho(f_{*}(\pi_{1}\Sigma))$
where $f_{*}$ is the induced homomorphism of $f:\Sigma\to M.$ }

\textsf{The hypotheses throughout here are: $\Gamma$ is a convex
cocompact and $(\Sigma,g)$ is negatively curved where $g=f^{*}h$
and $h$ is the given hyperbolic metric on $M$.}

\textsf{Notice that because $(\Sigma,g)$ is a compact negatively
curved surface, its universal covering $(\us,\widetilde{g})$ is a
pinched Hadamard manifold. Let $\Gamma_{\Sigma}$ denote  the group
of deck transformations of the covering $\us$ . Then we know $\Gamma_{\Sigma}\cong\pi_{1}\Sigma$
and $\Gamma_{\Sigma}\subset\mathrm{Isom}(\us)$. }

\begin{lem}
\label{lem: quasi-iso q}There exists a quasi-isometry $q:\widetilde{\Sigma}\to\mathrm{Conv}(\limitset)$,
where $\mathrm{Conv}(\limitset)$ is the convex hull of $\limitset$
in $\hy^{3}$. Moreover, $q$ extends to a bi-Hölder and $\Gamma$--equivariant
map between the boundaries.\end{lem}
\begin{proof}
Because $(\widetilde{\Sigma},\widetilde{g})$ and $(\ch,h)$ are both
pinched Hadamard manifold, and $\Gamma_{\Sigma}\in\mbox{\ensuremath{\mathrm{Isom}}}(\widetilde{\Sigma})$
acts cocompactly on $\widetilde{\Sigma}$ and $\Gamma$ acts convex
cocompactly on $\hy^{3}$, by Theorem \ref{thm:Svarc-milnor} (Švarc-Milnor
lemma), we know that there are quasi-isometries $q_{1}:\widetilde{\Sigma}\to C(\Gamma_{\Sigma},S')$
and $q_{2}:C(\Gamma,S)\to\ch$, where $C(\Gamma,S)$ is the Cayley
graph of $\Gamma=\langle S\rangle$. Because $\Gamma$ and $\Gamma_{\Sigma}$
are both isomorphic to $\pi_{1}\Sigma,$ by Švarc-Milnor lemma the
identity map $i:C(\Gamma_{\Sigma},S')\to C(\Gamma,S)$ is a quasi-isometry.
So, $q_{2}iq_{1}:\widetilde{\Sigma}\to\ch$ is the desired quasi-isometry.
The second assertion is a consequence of Proposition \ref{thm: Bourdon},
because the quasi-isometry $q=q_{2}iq_{1}$ extends to a bi-Hölder
map $q:\vbdy\widetilde{\Sigma}\to\limitset=\partial_{\infty}\mbox{Conv}(\limitset)$.
Lastly, by the construction of $q$, it is easy to see that $q$ is
$\Gamma$--equivariant.
\end{proof}

\textsf{Now, we have two different objects $\Gamma_{\Sigma}\subset\mathrm{Isom}(\us)$
and $\Gamma\subset\mathrm{Isom}(\hy^{3})$. Nevertheless, they are
the same as a group. Because $\pi_{1}\Sigma$ is finitely generated,
there are canonical isomorphisms between $\pi_{1}\Sigma$, $\Gamma_{\Sigma}$
and $\Gamma$, by sending generators to generators. Namely, $i_{\Sigma}:\pi_{1}\Sigma\to\Gamma_{\Sigma}$
and $i_{M}:\pi_{1}\Sigma\to\Gamma$. For brevity, we denote elements
in $\pi_{1}\Sigma$, $\Gamma_{\Sigma}$ and $\Gamma$ by $\gamma$,
$\gamma_{\Sigma}$ and $\gamma_{M}$ ,respectively, where $i_{\Sigma}(\gamma)=\gamma_{\Sigma}$
and $i_{M}(\gamma)=\gamma_{M}$.}

\begin{lem}
\label{Lem: closed geo to closed geo}The above quasi-isometry $q:\vbdy\widetilde{\Sigma}\to\limitset$
sends the attracting (repelling) limit point $\gamma_{\Sigma}^{+}$
($\gamma_{\Sigma}^{-}$) of the hyperbolic element $\gamma_{\Sigma}\in\Gamma_{\Sigma}\subset\mathrm{Isom}(\widetilde{\Sigma})$
to the attracting (repelling) limit point $\gamma_{M}^{+}$ ($\gamma_{M}^{-}$)
of $\gamma_{M}\in\Gamma\subset\mathrm{Isom}(\hy^{3})$.\end{lem}
\begin{proof}
Notice that the boundary map $q:\vbdy\widetilde{\Sigma}\to\limitset$
is an equivariant homeomorphism, and having an attracting (repelling)
point is a topological feature. Therefore, $q$ maps the attracting
point of $\gamma_{\Sigma}\in\Gamma_{\Sigma}$ to the attracting point
of $\gamma_{M}\in\Gamma$. 
\end{proof}

\textsf{Now we are ready to state and prove the main theorem. }

\begin{thm}[Main theorem]

\label{Thm: main-1} 

Let $f:\Sigma\to M$ be a $\pi_{1}-$injective immersion from a compact
surface $\Sigma$ to a hyperbolic 3-manifold $M$, and $\Gamma$ be
the copy of $\pi_{1}\Sigma$ in $\mathrm{Isom}(\hy^{3})$ induced
by the immersion $f$. Suppose $\Gamma$ is convex cocompact and $(\Sigma,f^{*}h)$
is negatively curved, then 
\[
C_{1}(\Sigma,M)\cdot\delta_{\Gamma}\leq h(\Sigma)\leq C_{2}(\Sigma,M)\cdot\delta_{\Gamma}.
\]
Here $h(\Sigma)$ is the topological entropy of the geodesic flow
on $T^{1}\Sigma$, $\delta_{\Gamma}$ is the critical exponent, and
$C_{1}(\Sigma,M)$, $C_{2}(\Sigma,M)$ are two geometric constants
smaller or equal to 1. Moreover, each equality holds if and only if
the marked length spectrum of $\Sigma$ is proportional to the marked
length spectrum of $M$, and the proportion is the ratio $\frac{\delta_{\Gamma}}{h(\Sigma)}.$

\end{thm}

\textsf{We will discuss the geometric meaning of $C_{1}(\Sigma,M)$
and $C_{2}(\Sigma,M)$ in detail in the next section. In this section
we shall focus on the proof and the rigidity phenomena coming from
the equality cases. }

\begin{proof} [Proof of the main theorem]

Let $\phi$ denote the geodesic flow on the unit tangent bundle of
$(\Sigma,g)$, i.e. $\phi:T^{1}\Sigma\to T^{1}\Sigma.$

\textbf{The first step }is to construct a Hölder reparametrization
function $F:T^{1}M\to\real_{>0}$ such that the topological entropy
$h_{F}$ of the reparametrized flow $\phi^{F}$ is the critical exponent
$\delta_{\Gamma}$ of $\Gamma$ in $\hy^{3}$. 

Recall the Busemann function $B_{\eta}^{h}(x,y):\partial_{\infty}\mathbb{H}^{3}\times\mathbb{H}^{3}\times\mathbb{H}^{3}\to\real$,
for $\eta\in\vbdy\hy^{3}$ and $x,y\in\hy^{3}$ is given by \textbf{
\[
B_{\eta}^{h}(x,y):=\lim_{z\to\eta}d_{h}(x,z)-d_{h}(y,z).
\]
}

Using the quasi-isometry $q$ defined in Lemma \ref{lem: quasi-iso q},
we define a map $c:\pi_{1}\Sigma\times\vbdy\widetilde{\Sigma}:\to\real$
by 
\begin{alignat*}{1}
c:\pi_{1}\Sigma\times\vbdy\widetilde{\Sigma} & \to\real\\
(\gamma,\xi) & \mapsto B_{q(\xi)}^{h}(f(o),\gamma^{-1}\cdot f(o)),
\end{alignat*}
for $o\in\widetilde{\Sigma}$. 

\textbf{Claim:} $c$ is a Hölder cocycle. 

pf. 
\begin{alignat*}{2}
c(\gamma_{1}\gamma_{2},\xi) & =B_{q(\xi)}^{h}(f(o),(\gamma_{1}\gamma_{2})^{-1}\cdot f(o))\\
 & =B_{q(\xi)}^{h}(f(o),(\gamma_{2}^{-1}\gamma_{1}^{-1})\cdot f(o))\\
 & =B_{q(\xi)}^{h}(f(o),\gamma_{2}^{-1}\cdot f(o))+B_{q(\xi)}^{h}(\gamma_{2}^{-1}\cdot f(o),(\gamma_{2}^{-1}\gamma_{1}^{-1})\cdot f(o))\\
 & =c(\gamma_{2},\xi)+B_{\gamma_{2}q(\xi)}^{h}(f(o),\gamma_{1}^{-1}\cdot f(o))\\
 & =c(\gamma_{2},\xi)+B_{q(\gamma_{2}\xi)}^{h}(f(o),\gamma_{1}^{-1}\cdot f(o)) & \mbox{by Lemma }\ref{Lem: closed geo to closed geo}\\
 & =c(\gamma_{2},\xi)+c(\gamma_{1},\gamma_{2}\xi).
\end{alignat*}

Therefore, $c$ is a cocycle. To see $c$ is Hölder, we first notice
that the boundary map $q:\vbdy\widetilde{\Sigma}\to\limitset\subset\vbdy\hy^{3}$
is bi-Hölder as we have discussed in the beginning of this section.
Moreover, we know that $\limitset$ embeds in $\vbdy\hy^{3}$ and
$B_{\eta}^{h}(x,y)$ is smooth on $\vbdy\hy^{3}$. Therefore, $c(\gamma,\cdot)$
is Hölder continuous on $\vbdy\widetilde{\Sigma}$, and we finish
the proof of this claim.

\medskip{}

Notice that the period $c(\gamma,\gamma_{\Sigma}^{+})=B_{q(\gamma_{\Sigma}^{+})}^{h}(f(o),\gamma^{-1}f(o))=l_{h}(\gamma))>0$
for all $[\gamma]\in[\pi_{1}\Sigma]$. Thus, $l_{c}(\gamma)=l_{h}(\gamma)$
for all $[\gamma]\in[\pi_{1}\Sigma]$, and we can easily see that
\[
h_{c}=\delta_{\Gamma}=\lim_{T\to\infty}\frac{1}{T}\log\#\{[\gamma]\in[\pi_{1}\Sigma];l_{h}(\gamma)\leq T\}<\infty.
\]
 Thus, by Theorem \ref{Thm: Sambarino}, there exists a positive Hölder
continuous maps $F_{c}$ on $T^{1}\Sigma$ such that the translation
flow defined by the Hölder cocycle $c$ is conjugated to the reparametrization
$\phi^{F_{c}}$ of the geodesic flow $\phi_{t}:T^{1}\Sigma\to T^{1}\Sigma$
by $F_{c}$. In particular, for all $[\gamma]\in[\pi_{1}\Sigma]$
\[
c(\gamma,\gamma_{\Sigma}^{+})=\int_{[\gamma]}F_{c}=l_{h}(\gamma),
\]
and the topological entropy of the flow $\phi^{F_{c}}$ is exactly
the exponential growth rate of $c$, i.e. $h_{F_{c}}=h_{c}.$

Notice that for the constant function $1$ on $T^{1}\Sigma$, we have
$l_{g}(\gamma)=\int_{[\gamma]}1$ for all $[\gamma]\in[\pi_{1}\Sigma]$.
Therefore, we have the pressure of the function $-h_{1}\cdot1$ is
zero, i.e. $P(-h_{1}\cdot1)=0$, where 
\[
h_{1}=\lim_{T\to\infty}\frac{1}{T}\log\#\{[\gamma]\in[\pi_{1}\Sigma];l_{g}(\gamma)\leq T\}
\]
is the topological entropy of the geodesic flow $\phi$ on $T^{1}\Sigma.$

From now on we denote $F_{c}$ by $F.$

\textbf{The second step} is to show that

\[
h(\Sigma)\leq\underset{c_{2}(\Sigma,M)}{\underbrace{\int F\mbox{d}\mu_{BM}}}\cdot h_{F},
\]
where $\mu_{BM}$ the Bowen-Margulis measure of the geodesic flow
$\phi:T^{1}\Sigma\to T^{1}\Sigma$.

Since 
\begin{alignat*}{1}
P(-h_{F}\cdot F) & =0=h(\mu_{-h_{F}F})-h_{F}\int F\mbox{d}\mu_{-h_{F}F}\\
P(-h(\Sigma)\cdot1) & =0=h(\mu_{BM})-h(\Sigma)\cdot\int1\mbox{d}\mu_{BM}=h(\mu_{BM})-h(\Sigma).
\end{alignat*}
where $\mu_{-h_{F}F}$ is the equilibrium state of $-h_{F}F$. Since
$\mu_{BM}\in\mathcal{M}^{\phi}$, by the variational principle we
have 
\[
P(-h_{F}\cdot F)=0\geq h(\mu_{BM})-h_{F}\int F\mbox{d}\mu_{BM}.
\]

Furthermore, 
\[
h_{F}\int F\mbox{d}\mu_{BM}\geq h(\mu_{BM})=h(\Sigma).
\]

\textbf{The third step} is to show the inequality 
\[
\underset{c_{1}(\Sigma,M)}{\underbrace{\int F\mbox{d}\mu_{-Fh_{F}}}}\cdot h_{F}\leq h(\Sigma).
\]

By Remark \ref{rem: pressure} , we have 
\begin{align*}
h(\Sigma) & \geq h(\mu_{-Fh_{F}})\\
\iff h(\Sigma)-h_{F}\int F\mbox{d}\mu_{-Fh_{F}} & \geq\underset{=0}{\underbrace{h(\mu_{-Fh_{F}})-h_{F}\int F\mbox{d}\mu_{-Fh_{F}}}}\\
\iff h(\Sigma) & \geq h_{F}\cdot\int F\mbox{d}\mu_{-Fh_{F}}.
\end{align*}

\textbf{The fourth step} is to show that $0\leq C_{1}(\Sigma,M)\leq1$
and $0\le C_{2}(\Sigma,M)\leq1.$

Because $C_{1}(\Sigma,M)=\int F\mbox{d}\mu_{-Fh_{F}}$ , $C_{2}(\Sigma,M)=\int F\mbox{d}\mu_{BM}$
and $F$ is positive, it is enough to show that $F$ could be chosen
to be smaller or equal than 1. 

\textbf{Claim:} $F\leq1$.

This is a consequence of Theorem \ref{Thm Nonnegative-Liv=000161ic-Theorem}.
For each conjugacy class $[\gamma]\in[\pi_{1}\Sigma]$ there exists
a unique closed geodesic $\tau_{\gamma}^{\Sigma}$ on $\Sigma$ such
that $l_{g}(\gamma)=l_{g}(\tau_{\gamma}^{\Sigma})$. Because $f$
is $\pi_{1}$--injective, $f$ maps $\tau_{\gamma}^{\Sigma}$ to a
closed curve $f(\tau_{\gamma}^{\Sigma})$ on $M$ which is in the
same free homotopy class generated by $[\gamma]$. More precisely,
let $\tau_{\gamma}^{M}$ denote the closed geodesic on $M$ in the
conjugacy class $[\gamma]$, then we know that $f(\tau_{\gamma}^{\Sigma})$
and $\tau_{\gamma}^{M}$ are in the same free homotopy class. Moreover,
because $g$ is the induced metric $f^{*}h$, we know that $(\Sigma,g)$
is Riemannian isometric to $(f(\Sigma),h)$. Thus, $l_{g}(\tau_{\gamma}^{\Sigma})=l_{h}(f(\tau_{\gamma}^{\Sigma}))$.
Therefore, $\forall$ $[\gamma]\in[\pi_{1}\Sigma]$, 
\begin{alignat*}{2}
l_{g}(\gamma) & = & l_{g}(\tau_{\gamma}^{\Sigma})=l_{h}(f(\tau_{\gamma}^{\Sigma}))\geq l_{h}(\tau_{\gamma}^{M})=l_{h}(\gamma).
\end{alignat*}
Therefore, for all $[\gamma]\in[\pi_{1}\Sigma]$ 
\[
\int_{[\gamma]}1=l_{g}(\gamma)\geq l_{h}(\gamma)=\int_{[\gamma]}F.
\]

By Theorem \ref{Thm Nonnegative-Liv=000161ic-Theorem}, we have $1-F$
is cohomologous to a nonnegative Hölder continuous function $H$,
and $H$ is unique up to cohomology. Thus, we have that $F\sim1-H$
and $1-H\leq1$. By choosing $F$ to be $1-H$, we now finish the
proof of this claim.

\medskip{}

\textbf{The fifth step} is to examine the equality cases. 

If $h(\Sigma)=h_{F}\int F\mbox{d}\mu_{-Fh_{F}}$, then $h(\Sigma)=h(\mu_{-Fh_{F}})$.
i.e., $\mu_{-Fh_{F}}$ is the equilibrium state of the constant function
$-h(\Sigma)\cdot1$. By the uniqueness part of Theorem \ref{Thm: unique eq state},
we have that $Fh_{F}$ is cohomologous to the constant $h(\Sigma)$,
i.e. $F\sim\frac{h(\Sigma)}{h_{F}}$. Similarly, if $h(\Sigma)=h_{F}\cdot\int F\mbox{d}\mu_{BM}$,
then $\mu_{BM}=\mu_{-h_{F}F}$. Hence, again, $h(\Sigma)\sim F\cdot h_{F}$,
i.e.$F\sim\frac{h(\Sigma)}{h_{F}}$.

\end{proof}

\begin{cor}
\label{cor:rigidity}If $h(\Sigma)=\delta_{\Gamma}$, then $\Sigma$
is a totally geodesic submanifold in $M$.
\end{cor}

\begin{proof}
Notice $h(\Sigma)=\delta_{\Gamma}$ implies $F=1$. This means that
the length of each closed geodesic on $\Sigma$ has the same length
with the corresponding closed geodesic on $M$. Furthermore, we know
that the closed geodesics in $\Sigma$ are dense. In the sense that,
for any point $p\in\Sigma$, the set of tangent vectors $v\in T_{p}\Sigma$
such that the exponential map $\exp_{p}tv$ gives a closed geodesic
is dense in $T{}_{p}\Sigma$. Therefore, the shape operator $S_{g}$
is zero when evaluating on this dense subset of vectors on $T_{p}\Sigma$.
By the continuity of the shape operator $S_{g}$, we have $S_{g}\equiv0$.
Therefore $\Sigma$ is totally geodesic in $M$. 
\end{proof}

\section{Geometric meaning of $C_{1}(\Sigma,M)$ and $C_{2}(\Sigma,M)$\label{sec:Geodesic-Stretch}}

\textsf{Throughout this section we keep the same setting as in the
last section: $f:\Sigma\to M$ is a $\pi_{1}-$injective immersion
from a compact surface $\Sigma$ to a hyperbolic 3-manifold $M$,
$\pi_{1}\Sigma\cong\Gamma$ is the subgroup of $\mathrm{Isom}(\hy^{3})$
induced by $f$ and $h$ is the given hyperbolic metric on $M$, and
assuming that $(\Sigma,f^{*}h)$ is negatively curved and $\Gamma$
is convex cocompact. In this section we will discuss the geometric
meaning of these two constants. When $f$ is an immersion, these two
constants could be regarded as averages of lengths of closed geodesics;
moreover when $f$ is an embedding, we show that these two constants
are exactly the geodesic stretches defined imitating \cite{Anonymous:1995bn}. }

\subsection{Immersion}

\textsf{First, using the equidistribution property of Gibbs measures,
we can understand $C_{1}(\Sigma,M)$ and $C_{2}(\Sigma,M)$ by averages
of the length of closed geodesics with respect to different metrics.}

\begin{thm}\label{thm: c1c2}

Let $\mu_{BM}$ be the Bowen-Margulis measure of the geodesic flow
$\phi:T^{1}\Sigma\to T^{1}\Sigma$ and $\mu_{-h_{F}F}$ be the Gibbs
measure for $-h_{F}F$ defined in Theorem \ref{Thm: main-1}. Then 

\[
C_{2}(\Sigma,M)=\int F\mbox{d}\mu_{BM}=\lim_{T\to\infty}\frac{1}{\#R_{T}(g)}\sum_{[\gamma]\in R_{T}(g)}\frac{l_{h}(\gamma)}{l_{g}(\gamma)}=\lim_{T\to\infty}\frac{{\displaystyle \sum_{[\gamma]\in R_{T}(g)}l_{h}(\gamma)}}{{\displaystyle \sum_{[\gamma]\in R_{T}(g)}l_{g}(\gamma)}}
\]

and

\[
C_{1}(\Sigma,M)=\int F\mbox{d}\mu_{-h_{F}F}=\left(\lim_{T\to\infty}\frac{1}{\#R_{T}(h)}\sum_{[\gamma]\in R_{T}(h)}\frac{l_{g}(\gamma)}{l_{h}(\gamma)}\right)^{-1}=\lim_{T\to\infty}\frac{{\displaystyle \sum_{[\gamma]\in R_{T}(h)}l_{h}(\gamma)}}{{\displaystyle \sum_{[\gamma]\in R_{T}(h)}l_{g}(\gamma)}}
\]
where 
\[
R_{T}(g):=\{[\gamma]\in[\pi_{1}\Sigma]:\mbox{ }l_{g}(\gamma)\leq T\}\mbox{ and }R_{T}(h):=\{[\gamma]\in[\pi_{1}\Sigma]:\mbox{ }l_{h}(\gamma)\leq T\}.
\]

\end{thm}

\begin{proof}
This is a consequence of the equidistribution theorem (Theorem \ref{Thm: Equidistribution}).

By Theorem \ref{Thm: Equidistribution}, we have 
\begin{alignat*}{1}
C_{2}(\Sigma,M)=\int F\mbox{d}\mu_{BM} & =\lim_{T\to\infty}\frac{1}{\#R_{T}(1)}\sum_{\tau\in R_{T}(1)}\frac{\langle\delta_{\tau},F\rangle}{\langle\delta_{\tau},1\rangle}=\lim_{T\to\infty}\frac{{\displaystyle \sum_{\tau\in R_{T}(1)}\langle\delta_{\tau},F\rangle}}{{\displaystyle \sum_{\tau\in R_{T}(1)}\langle\delta_{\tau},1\rangle}}.
\end{alignat*}
Notice that every closed orbit $\tau$ of the geodesic flow $\phi$
on $T^{1}\Sigma$ corresponds to a unique conjugacy class $[\gamma^{\tau}]$
of $\pi_{1}\Sigma$, and vice versa. Moreover, the period of $\tau$
is the length of $\gamma^{\tau}$ on $\Sigma$, i.e. 
\[
l_{h}(\gamma^{\tau})=\langle\delta_{\tau},F\rangle,\mbox{ }l_{g}(\gamma^{\tau})=\langle\delta_{\tau},1\rangle.
\]
 Since there is an one-to-one correspondence between $R_{T}(1)$ and
$R_{T}(g),$ we can rewrite the equation above by

\begin{alignat*}{1}
C_{2}(\Sigma,M)=\int F\mbox{d}\mu_{BM} & =\lim_{T\to\infty}\frac{1}{\#R_{T}(g)}\sum_{[\gamma]\in R_{T}(g)}\frac{l_{h}(\gamma)}{l_{g}(\gamma)}=\lim_{T\to\infty}\frac{{\displaystyle \sum_{[\gamma]\in R_{T}(g)}l_{h}(\gamma)}}{{\displaystyle \sum_{[\gamma]\in R_{T}(g)}l_{g}(\gamma)}}.
\end{alignat*}

For the other equation, by Theorem \ref{thm: bowen's formula}, we
know that $\mu_{\phi^{F}}=\widehat{F.\mu_{-h_{F}F}}$. Therefore 
\[
\mu_{\phi^{F}}(\frac{1}{F})=\widehat{F.\mu_{-h_{F}F}}(\frac{1}{F})=\frac{\int(\frac{1}{F})\cdot F\mbox{d}\mu_{-Fh_{F}}}{\int F\mbox{d}\mu_{-Fh_{F}}}=\frac{1}{\int F\mbox{d}\mu_{-Fh_{F}}}.
\]
By Theorem \ref{Thm: Equidistribution}, we have 
\[
\mu_{\phi^{F}}(\frac{1}{F})=\lim_{T\to\infty}\frac{1}{\#R_{T}(F)}\sum_{\tau'\in R_{T}(F)}\frac{\langle\delta_{\tau'}^{F},\frac{1}{F}\rangle}{\langle\delta_{\tau'}^{F},1\rangle}=\lim_{T\to\infty}\frac{{\displaystyle \sum_{\tau'\in R_{T}(F)}\langle\delta_{\tau'}^{F},\frac{1}{F}\rangle}}{{\displaystyle \sum_{\tau'\in R_{T}(F)}\langle\delta_{\tau'}^{F},1\rangle}}.
\]
Notice that for a closed geodesic $\tau'$ of the geodesic flow $\phi:T^{1}\Sigma\to T^{1}\Sigma$,
$\langle\delta_{\tau'}^{F},\frac{1}{F}\rangle=\int_{0}^{l_{g}(\tau')}\frac{1}{F(\phi_{t})}\cdot F(\phi_{t})dt=l_{g}(\tau')$
and similarly $\langle\delta_{\tau'}^{F},F\rangle=\int_{0}^{l_{g}(\tau')}F(\phi_{t})dt=l_{h}(\tau')$.
By the one-to-one correspondence between closed orbit $\tau'$ and
conjugacy class $[\gamma^{\tau'}]$, we have an one-to-one correspondence
between $R_{T}(F)$ and $R_{T}(h)$. 

Hence, we have the following equation:
\[
C_{1}(\Sigma,M)=\int F\mbox{d}\mu_{-h_{F}F}=\left(\mu_{\phi^{F}}(\frac{1}{F})\right)^{-1}=\left(\lim_{T\to\infty}\frac{1}{\#R_{T}(h)}\sum_{[\gamma]\in R_{T}(h)}\frac{l_{g}(\gamma)}{l_{h}(\gamma)}\right)^{-1}=\lim_{T\to\infty}\frac{{\displaystyle \sum_{[\gamma]\in R_{T}(h)}l_{h}(\gamma)}}{{\displaystyle \sum_{[\gamma]\in R_{T}(h)}l_{g}(\gamma)}}
\]

\end{proof}

\begin{rem}
:\begin{itemize}

\item[1.]From this result, although we don't understand the measure
$\mu_{-Fh_{F}}$ much, we still see that the integral $\int F\mbox{d}\mu_{-Fh_{F}}$
is exactly $\int G\mbox{d}\mu_{BM}^{'}$ where $G$ is the reparametrization
function that we get when we reparametrize the geodesic flow $\psi$
on $T^{1}M$ to conjugate the geodesic flow $\phi_{t}$ on $T^{1}\Sigma$,
and $\mu_{BM}^{'}$ is the Bowen-Margulis measure of $\psi$.

\item[2.] From above expression of $C_{1}(\Sigma,M)$ and $C_{2}(\Sigma,M)$,
we can also see $C_{1}(\Sigma,M)\leq1$ and $C_{2}(\Sigma,M)\leq1$.
It is because for each $[\gamma]\in[\pi_{1}\Sigma]$, we know $l_{g}(\gamma)\geq l_{h}(\gamma)$
(cf. step 4 in the proof of Theorem \ref{Thm: main-1}). 

\end{itemize}
\end{rem}

\subsection{Embedding}

\textsf{In this subsection, we will assume that $f:\Sigma\to M$ is
an embedding. To state our results more precisely and to put it in
context, we first introduce the geodesic stretch and discuss the relation
between the geodesic stretch, $C_{1}(\Sigma,M)$ and $C_{2}(\Sigma,M)$. }

\textsf{Notice that, we can lift $f:\Sigma\to M$ to an embedding
between their universal coverings, i.e. $\widetilde{f}:\us\to\widetilde{M}=\hy^{3}$.
Moreover, one can easily check that this lifting is $\pi_{1}\Sigma$-equivariant.
Specifically, for each $\gamma\in\pi_{1}\Sigma$, let $\gamma_{\Sigma}\in\Gamma_{\Sigma}$
and $\gamma_{M}\in\Gamma$ be the corresponding element of $\gamma$
in the deck transformation groups $\Gamma_{\Sigma}\subset\Isom\us$
and $\Gamma\subset\Isom(\hy^{3})$, respectively. Then for each $\widetilde{x}\in\us$
we have 
\[
\widetilde{f}(\gamma\cdot\widetilde{x}):=\widetilde{f}(\gamma_{\Sigma}(\widetilde{x}))=\gamma_{M}(\widetilde{f}(\widetilde{x}))=:\gamma\cdot\widetilde{f}(\widetilde{x}).
\]
 Using this embedding $\widetilde{f}:\widetilde{\Sigma}\to\hy^{3}$
we can define a tangent map $\mathbf{f}:T^{1}\widetilde{\Sigma}\to T^{1}\hy^{3}$
by 
\[
\mathbf{f}:(\widetilde{x}_{0},w)\mapsto(\widetilde{f}(\widetilde{x_{0}}),d\widetilde{f}_{\widetilde{x_{0}}}(w))
\]
 where $\widetilde{x_{0}}\in\widetilde{\Sigma}$ and $w$ is a unit
vector on the tangent plane $T_{\widetilde{x}_{0}}\widetilde{\Sigma}$.
Notice that $\pi_{1}\Sigma$ acts on $T^{1}\us$ and $T^{1}\hy^{3}$
in an obvious way. Thus $\mathbf{f}$ is also $\pi_{1}\Sigma$--equivariant.
More precisely, $\gamma\cdot\mathbf{f}(\widetilde{x_{0}},w)=(\gamma\cdot\widetilde{f}(\widetilde{x_{0}}),d\widetilde{f}_{\widetilde{x_{0}}}(w))=(\widetilde{f}(\gamma\cdot x_{0}),d\widetilde{f}_{\widetilde{x_{0}}}(w))=\mathbf{f}(\gamma\cdot(\widetilde{x_{0}},w))$.}

\medskip{}

\textsf{The following lemma depicts a key feature of the embedding
$f:\Sigma\to M$.}

\begin{lem}
\label{Lemma: embedding quasi-iso}$(\widetilde{\Sigma},d_{g})$ is
quasi-isometric to $(\widetilde{f}(\widetilde{\Sigma}),d_{h})\subset(\hy^{3},d_{h})$
where $d_{g}$ is the distance on $\widetilde{\Sigma}$ induced by
$g$ and $d_{h}$ is the hyperbolic distance on $\hy^{3}$.\end{lem}
\begin{proof}
Because $\widetilde{f}$ is an embedding and $\pi_{1}\Sigma-$equivariant,
we know that $(\widetilde{f}(\widetilde{\Sigma}),d_{h})$ is a proper
geodesic space and $\Gamma\in\Isom(\widetilde{f}(\widetilde{\Sigma}))\subset\Isom(\hy^{3})$
acts properly discontinuously and compactly on $\widetilde{f}(\widetilde{\Sigma})$.
Hence, by Theorem \ref{thm:Svarc-milnor} (Švarc-Milnor lemma), $(\widetilde{\Sigma},d_{g})$
is quasi-isometric $(\widetilde{f}(\widetilde{\Sigma}),d_{h})$. (Because
$(\widetilde{\Sigma},d_{g})$ and $(\widetilde{f}(\widetilde{\Sigma}),d_{h})$
are both quasi-isometric to the Cayley graph of $\pi_{1}\Sigma$ with
a word metric.)
\end{proof}

\begin{defn}
For all $v\in T^{1}\widetilde{\Sigma}$ and $t>0,$ we define
\[
a(v,t):=d_{h}(\pi\circ\mathbf{f}(v),\pi\circ\mathbf{f}\circ\widetilde{\phi_{t}}(v)),
\]
where $\pi:T^{1}\widetilde{\Sigma}\to\widetilde{\Sigma}$ is the natural
projection and $\widetilde{\phi}$ is the lift of $\phi.$ 
\end{defn}

\begin{rem}
$a(v,t)$ is $\pi_{1}\Sigma$--invariant, because $\mathbf{f}$ is
$\pi_{1}\Sigma$--equivariant and $\pi_{1}\Sigma$ is acting on $\widetilde{\Sigma}$
via $\Gamma_{\Sigma}\subset\Isom(\us)$.
\end{rem}

\begin{lem}
For all $v\in T^{1}\widetilde{\Sigma}$ and $t_{1},t_{2}>0,$

\[
a(v,t_{1}+t_{2})\leq a(v,t_{1})+a(\widetilde{\phi_{t_{1}}}(v),t_{2}).
\]
\end{lem}
\begin{proof}
It's easy consequence of the triangle inequality of $d_{h}$.
\end{proof}

\textsf{The following corollary is a consequence of Kingman's sub-additive
ergodic theorem \cite{Kingman:1973uo}.}
\begin{cor}
\label{Cor subadd}Let $\mu$ be a $\phi_{t}-$invariant probablity
measure on $T^{1}\Sigma$. Then for $\mu-a.e.$ $v\in T^{1}\Sigma$
\[
I_{\mu}(\Sigma,M,v):=\lim_{t\to\infty}\frac{a(v,t)}{t},
\]
 and defines a $\mu-$integrable function on $T^{1}\Sigma$, invariant
under the geodesic flow $\phi_{t}$.\end{cor}
\begin{proof}
Kingman's original theorem works on measure preserving transformations;
nevertheless, it works on flow as well. More precisely, for flows,
we consider the time one map to be the measure preserving transformation.
Thus, the only condition that we need is ${\displaystyle \sup\{a(v,t);v\in T^{1}\widetilde{\Sigma},\mbox{ }0\leq t\leq1\}\in L^{1}(\mu)}$.
Notice that we can always translate $v\in T^{1}\widetilde{\Sigma}$
to a fixed copy of $T^{1}\Sigma$ in $T^{1}\widetilde{\Sigma}$ by
a deck transformation. Because $T^{1}\Sigma$ is compact and $a(v,t)$
is $\pi_{1}\Sigma$--invariant, we have that $\sup\{a(v,t);v\in T^{1}\widetilde{\Sigma},\mbox{ }0\leq t\leq1\}$
is bounded. 
\end{proof}

\medskip{}

\textsf{From the above corollary, we can define the geodesic stretch
as the following. }

\begin{defn}
The \textit{geodesic stretch} $I_{\mu}(\Sigma,M)$ of $\Sigma$ relative
to $M$ and a $\phi_{t}-$invariant probablity measure $\mu$, i.e.
$\mu\in\mathcal{M}^{\phi}$, is defined as 
\[
I_{\mu}(\Sigma,M):=\int_{T^{1}\Sigma}I_{\mu}(\Sigma,M,v)\mbox{d}\mu.
\]
 
\end{defn}

\begin{rem}
If $\mu\in\mathcal{M}^{\phi}$ is ergodic, then $I_{\mu}(\Sigma,M)={\displaystyle \lim_{t\to\infty}\frac{a(v,t)}{t}}$
for $\mu$--a.e. $v\in T^{1}\Sigma$.
\end{rem}

\textsf{Since $f:(\us,d_{g})\to(f(\us),d_{h})$ is a quasi-isometry,
by Theorem \ref{thm: Bourdon} we know that $f$ extends to a bi-Hölder
map between $\vbdy\us$ and $\vbdy f(\us)=\limitset$. By the same
discussion as in Lemma \ref{Lem: closed geo to closed geo}, we know
that $f$ maps the attracting (repelling) fixed point $\gamma_{\Sigma}^{+}$
($\gamma_{\Sigma}^{-}$) of $\gamma_{\Sigma}\in\Gamma_{\Sigma}$ to
the corresponding attracting (repelling) fixed point $\gamma_{M}^{+}$
($\gamma_{M}^{-}$) of $\gamma_{M}\in\Gamma$. }

\textsf{Moreover, each conjugacy class $[\gamma]\in[\pi_{1}\Sigma]$
corresponds to a unique closed geodesic $\tau_{\gamma}^{\Sigma}$
on $\Sigma$ and $\tau_{\gamma}^{M}$ on $M$ , and $\tau_{\gamma}^{\Sigma}$
also corresponds to the unique geodesic $\widetilde{\tau_{\gamma}^{\Sigma}}$
connecting $\gamma_{\Sigma}^{-}$ and $\gamma_{\Sigma}^{+}$ on $\vbdy\us.$
Notice that $\widetilde{f}(\gamma_{\Sigma}^{-})=\gamma_{M}^{-}$ and
$\widetilde{f}(\gamma_{\Sigma}^{+})=\gamma_{\Sigma}^{+}$ on $\vbdy\widetilde{f}(\us)=\limitset\subset\vbdy\hy^{3}$,
so $f(\widetilde{\tau_{\gamma}^{\Sigma}})$ is a quasi-geodesic on
$\hy^{3}$ within a bounded Hausdorff distance from the geodesic $\widetilde{\tau_{\gamma}^{M}}$
on $\hy^{3}$, where $\widetilde{\tau_{\gamma}^{M}}$ is the geodesic
on $\mathrm{Conv}(\limitset)\subset\hy^{3}$ connecting $\gamma_{M}^{-}$
and $\gamma_{M}^{+}$ on $\limitset$.}

\begin{lem}
\label{Lem: F =00003D I} If $\mu\in\mathcal{M^{\phi}}$ and ergodic,
then there exists a sequence of conjugacy classes $\{[\gamma_{n}]\}\subset[\pi_{1}\Sigma]$,
i.e. closed geodesics, such that 
\[
\int F\mbox{d}\mu=\lim_{n\to\infty}\frac{l_{h}(\gamma_{n})}{l_{g}(\gamma_{n})},
\]
where $F$ is the reparametrization function defined in Theorem \ref{Thm: main-1}. 
\end{lem}

\begin{proof}

First, by the sub-additive ergodic theorem we know that for $\mu-a.e.$
$v\in T^{1}\Sigma$
\begin{equation}
\lim_{t\to\infty}\frac{a(v,t)}{t}=I_{\mu}(\Sigma,M).\label{eq: subbadd ergodic}
\end{equation}
By the Birkhoff ergodic theorem we have for $\mu-a.e.$ $v\in T^{1}\Sigma$
\begin{equation}
\lim_{t\to\infty}\frac{1}{t}\int_{0}^{t}F(\phi_{s}v)\mbox{d}s=\int F\mbox{d}\mu.\label{eq:Birkhoff ergodic}
\end{equation}

We define two sets 
\begin{alignat*}{1}
A:= & \{v\in T^{1}\Sigma:v\mbox{ satisfies }(\ref{eq: subbadd ergodic})\}\\
B:= & \{v\in T^{1}\Sigma:v\mbox{ satisfies }(\ref{eq:Birkhoff ergodic})\}.
\end{alignat*}

Since $A$ and $B$ are both full $\mu$-measure, we have $A\cap B\neq\emptyset$.

Pick $v\in A\cap B$, and $\varepsilon_{n}\searrow0$ as $n\to\infty$.
By the Anosov Closing Lemma (Theorem \ref{Thm Anosov Closing lemma}),
for each $\varepsilon_{n}$, there exists $\delta_{n}=\delta_{n}(\varepsilon_{n})$
such that for $v\in T^{1}\Sigma$ and $T_{n}=T_{n}(\delta_{n})>0$
satisfying $D_{g}(\phi_{T_{n}}(v),v)<\varepsilon_{n}$, then there
exists $w_{n}\in T^{1}\Sigma$ generates a periodic orbit $\tau_{n}^{\Sigma}$
on $\Sigma$ of period $l_{g}(\tau_{n}^{\Sigma})=T_{n}'$ such that
$\left|T_{n}-T_{n}'\right|<\varepsilon_{n}$ and $D_{g}(\phi_{s}(v),\phi_{s}(w_{n}))<\varepsilon_{n}$
for all $s\in[0,T_{n}]$. 

Furthermore, because the geodesic flow $\phi_{t}$ on $T^{1}\Sigma$
is a transitive Anosov flow and $T^{1}\Sigma$ is compact, by the
Poincaré recurrent theorem, for each $\delta_{n}$ given as above,
we can pick $T_{n}$ to be the $n$-th return time of the flow $\phi_{t}$
to the set $B_{\delta_{n}}(v)$, i.e. $D_{g}(\phi_{T_{n}}(v),v)<\delta_{n}$
for each $n$.

Suppose $\tau_{n}^{\Sigma}$ corresponds to $[\gamma_{n}]\in[\pi_{1}\Sigma]$,\textsf{
}then since $\mu$ is ergodic, by the Birkhoff ergodic theorem we
have 
\[
{\color{red}}\int_{T^{1}\Sigma}F\mbox{d}\mu=\lim_{T\to\infty}\frac{1}{T}\int_{_{0}}^{T}F(\phi_{t}v)\mbox{d}t.
\]

\textbf{Claim:} ${\displaystyle \int F\mbox{d}\mu=\lim_{n\to\infty}\frac{\int_{\gamma_{n}}F}{l_{g}(\gamma_{n})}}$.

pf. Notice that 
\[
\frac{1}{l_{g}(\gamma_{n})+\varepsilon_{n}}\int_{0}^{l_{g}(\gamma_{n})-\varepsilon_{n}}F(\phi_{t}v)\leq\frac{1}{t_{n}}\int_{0}^{t_{n}}F(\phi_{t}v)\leq\frac{1}{l_{g}(\gamma_{n})-\varepsilon_{n}}\int_{0}^{l_{g}(\gamma_{n})+\varepsilon_{n}}F(\phi_{t}v).
\]

Because $F$ is Hölder, we know that $\left|F(\phi_{t}v)-F(\phi_{t}w_{n})\right|\leq C\cdot D_{g}(\phi_{t}v,\phi_{t}w_{n})^{\alpha}\leq C\cdot\varepsilon_{n}^{\alpha}$.

When $n$ is big enough such that $l_{g}(\gamma_{n})>2\varepsilon_{n}$
(notice that $\varepsilon_{n}\searrow0$ and $l_{g}(\gamma_{n})\nearrow\infty)$
, we have 

\begin{align*}
\left|\frac{1}{t_{n}}\int_{0}^{t_{n}}F(\phi_{t}v)-\frac{1}{l_{g}(\gamma_{n})}\int_{0}^{l_{g}(\gamma_{n})}F(\phi_{t}w_{n})\right| & \leq\frac{l_{g}(\gamma_{n})\int_{0}^{l_{g}(\gamma)}\left|F(\phi_{t}v)-F(\phi_{t}w_{n})\right|\mbox{d}t+2\cdot l_{g}(\gamma_{n})\cdot\varepsilon_{n}\cdot\left\Vert F\right\Vert _{\infty}}{l_{g}(\gamma_{n})\cdot(l_{g}(\gamma_{n})-\varepsilon_{n})}\\
 & \leq\frac{1}{l_{g}(\gamma_{n})-\varepsilon_{n}}\left(l_{g}(\gamma_{n})\cdot C\cdot\varepsilon_{n}^{\alpha}+2\varepsilon_{n}\cdot\left\Vert F\right\Vert _{\infty}\right)\\
 & \leq2C\cdot\varepsilon_{n}^{\alpha}+\frac{2\varepsilon_{n}}{l_{g}(\gamma_{n})-\varepsilon_{n}}\cdot\left\Vert F\right\Vert _{\infty}.
\end{align*}

So, we finish the proof of this claim.

Moreover, from the construction of $F$, $\forall[\gamma_{n}]\in[\pi_{1}\Sigma]$
we have 
\[
\int_{[\gamma_{n}]}F=l_{h}(\gamma_{n}).
\]
 Therefore, 
\[
\int F\mbox{d}\mu=\lim_{n\to\infty}\frac{\int_{\gamma_{n}}F}{l_{g}(\gamma_{n})}=\lim_{n\to\infty}\frac{l_{h}(\gamma_{n})}{l_{g}(\gamma_{n})}.
\]

\end{proof}

\begin{lem}
\label{Lem: equidistribution}There exists a sequence of conjugacy
classes $\{[\gamma_{n}]\}\subset[\pi_{1}\Sigma]$, i.e. closed geodesics,
such that 
\[
\lim_{n\to\infty}\frac{l_{h}(\gamma_{n})}{l_{g}(\gamma_{n})}=I_{\mu}(\Sigma,M).
\]

\end{lem}

\begin{proof}
Choose the $[\gamma_{n}]\in[\pi_{1}\Sigma]$ to be the sequence $\{[\gamma_{n}]\}$
we found in Lemma \ref{Lem: F =00003D I}.

\textbf{Claim:} 
\[
\lim_{n\to\infty}\frac{a(w_{n},l_{g}(\gamma_{n}))}{l_{g}(\gamma_{n})}=\lim_{n\to\infty}\frac{l_{h}(\gamma_{n})}{l_{g}(\gamma_{n})}
\]

pf. By definition, 
\[
a(w_{n},l_{g}(\gamma_{n})):=d_{h}(\pi\circ\mathbf{f}\circ w_{n},\pi\circ\mathbf{f}\circ\widetilde{\phi}_{l_{g}(\gamma_{n})}w_{n}).
\]

For such $[\gamma_{n}]\in[\pi_{1}\Sigma]$, let $\tau_{n}^{\Sigma}$
and $\tau_{n}^{M}$denote the corresponding closed geodesics on $\Sigma$
and $M$, and $\widetilde{\tau_{n}^{\Sigma}}$ and $\widetilde{\tau_{n}^{M}}$
denote their lifting on $\us$ and $\mathrm{Conv}(\limitset)$, respectively.
Then we know that $\widetilde{f}(\widetilde{\tau_{n}^{\Sigma}})$
and $\widetilde{\tau_{n}^{M}}$ are at most Hausdorff distance $R$
from each other. Therefore we can choose $x_{n}\in\widetilde{\tau_{n}^{M}}$
such that $d_{h}(\pi w_{n},x_{n})<R$. Because $d_{h}$ is $\Gamma-$invariant,
$\widetilde{f}:\widetilde{\Sigma}\to\hy^{3}$ is an embedding, and
$\pi\circ\mathbf{f}\circ w_{n}$ and $\pi\circ\mathbf{f}\circ\phi_{l_{g}(\gamma_{n})}w_{n}$
project to the same point on $\Sigma$, we have $d_{h}(\gamma_{n}\cdot x_{n},\pi\circ\mathbf{f}\circ\widetilde{\phi}_{l_{g}(\tau_{n})}w_{n})=d_{h}(\pi\circ\mathbf{f}\circ w_{n},x_{n})<R$.
Hence, by the triangle inequality 
\begin{alignat*}{1}
{\color{black}\left|d_{h}(\pi\circ\mathbf{f}\circ w_{n},\pi\circ\mathbf{f}\circ\widetilde{\phi}_{l_{g}(\tau_{n})}w_{n})-\underset{=l_{h}(\tau_{n})}{\underbrace{d_{h}(x_{n},\gamma_{n}\cdot x_{n})}}\right|} & {\color{black}\leq\underset{\leq R}{\underbrace{d_{h}(\pi\circ\mathbf{f}\circ w_{n},x_{n})}}+\underset{\leq R}{\underbrace{d_{h}(\gamma_{n}\cdot x_{n},\pi\circ\mathbf{f}\circ\widetilde{\phi}_{l_{g}(\tau_{n})}w_{n})}}}\\
{\color{black}} & {\color{black}=2R.}
\end{alignat*}
Therefore, 
\[
\lim_{n\to\infty}\frac{l_{h}(\gamma_{n})}{l_{g}(\gamma_{n})}=\lim_{n\to\infty}\frac{l_{h}(\gamma_{n})-2R}{l_{g}(\gamma_{n})}\leq\lim_{n\to\infty}\frac{a(w_{n},l_{g}(\gamma_{n}))}{l_{g}(\gamma_{n})}\leq\lim_{n\to\infty}\frac{l_{h}(\gamma_{n})+2R}{l_{g}(\gamma_{n})}=\lim_{n\to\infty}\frac{l_{h}(\gamma_{n})}{l_{g}(\gamma_{n})},
\]
 and we finish the proof of this claim.

\textbf{Claim:} 
\[
I_{\mu}(\Sigma,M)=\lim_{t\to\infty}\frac{a(v,t)}{t}=\lim_{n\to\infty}\frac{l_{h}(\gamma_{n})}{l_{g}(\gamma_{n})}.
\]

pf. Pick the $t_{n}$ as we mentioned in the first paragraph. Then
\begin{alignat*}{2}
\left|a(v,t_{n})-a(w_{n},l_{g}(\gamma_{n}))\right| & \leq\left|d_{h}(\pi\circ\mathbf{f}\circ v,\pi\circ\mathbf{f}\circ\widetilde{\phi}_{t_{n}}v)-d_{h}(\pi\circ\mathbf{f}\circ w_{n},\pi\circ\mathbf{f}\circ\widetilde{\phi}_{t_{n}}w_{n})\right|\\
 & \leq d_{h}(\pi\circ\mathbf{f}\circ v,\pi\circ\mathbf{f}\circ w_{n})+d_{h}(\pi\circ\mathbf{f}\circ\widetilde{\phi}_{l_{g}(\gamma_{n})}w_{n},\pi\circ\mathbf{f}\circ\widetilde{\phi}_{t_{n}}v) & \mbox{triangle ineq.}\\
 & \leq C\cdot\left(d_{g}(\pi\circ v,\pi\circ w_{n})+d_{g}(\pi\circ\widetilde{\phi}_{l_{g}(\gamma_{n})}w_{n},\pi\circ\widetilde{\phi}_{t_{n}}v)\right)+2L & \qquad\mathrm{quasi-iso.\mbox{ }(Lem.\mbox{ }}\ref{Lemma: embedding quasi-iso})\\
 & \leq C\cdot(\delta_{2}+\varepsilon)+2L, & \mbox{Anosov closing lemma}
\end{alignat*}
where $C$ and $L$ are the quasi-isometry constants only depending
on the embedding $f:\Sigma\to M$.

Therefore, 
\[
\lim_{t_{n}\to\infty}\frac{a(v,t_{n})}{t_{n}}=\lim_{n\to\infty}\frac{a(w_{n},l_{g}(\gamma_{n}))}{l_{g}(\gamma_{n})}=\lim_{n\to\infty}\frac{l_{h}(\gamma_{n})}{l_{g}(\gamma_{n})}.
\]

\end{proof}

\begin{thm}
\label{Thm: geodesic stetch}Suppose that $\Sigma$ is a compact and
negatively curved surface embedded in a hyperbolic 3-manifold $M$
as in Theorem \ref{Thm: main-1}. Assuming that $\Gamma$ is convex
cocompact, then 
\begin{alignat*}{1}
C_{1}(\Sigma,M) & =I_{\mu}(\Sigma,M),\\
C_{2}(\Sigma,M) & =I_{\mu_{BM}}(\Sigma,M),
\end{alignat*}

where $\mu$ is a $\phi-$invariant Gibbs measure and $\mu_{BM}$
is the Bowen-Margulis measure of the geodesic flow $\phi_{t}$ on
$T^{1}\Sigma$.
\end{thm}

\begin{proof}
It's a consequence of the above two lemmas, because $\mu_{BM}$ and
$\mu=\mu_{-h_{F}F}$ are Gibbs measures, which are, in particular,
ergodic measures of the flow $\phi_{t}$.
\end{proof}

\begin{rem}
Theorem \ref{Thm: geodesic stetch} also indicates that $C_{1}(\Sigma,M)$,
$C_{2}(\Sigma,M)\leq1$, because $a(v,t)\leq t$ for all $t>0$ and
$v\in T^{1}\Sigma$. 
\end{rem}

\section{Applications to immersed minimal surfaces in hyperbolic 3--manifolds\label{sec:Applications-to-Taubes'}}

\subsection{Immersed minimal surfaces in hyperbolic 3--manifolds}

\textsf{In what follows, $M$ denotes a hyperbolic 3--manifold equipped
with a hyperbolic metric $h$ and $\Sigma$ is a compact, 2--dimension
manifolds with negative Euler characteristic. Recall that $f:\Sigma\to M$
is called a minimal immersion if $f$ is an immersion and the its
mean curvature $H$ vanishes identically. }

\textsf{Let $g=f^{\ast}h$ denote the induced metric on $\Sigma$
via the immersion $f$. By the Gauss equation, when $f:\Sigma\to M$
is a minimal immersion, the Gaussian curvature $K_{g}\leq-1$. }

\textsf{So, applying the Theorem \ref{Thm: main-1} to this case,
we have the following corollary. }

\begin{cor}
\label{Cor: main}Let $f:\Sigma\to M$ be a $\pi_{1}$--injective
minimal immersion from a compact surface $\Sigma$ to a hyperbolic
3--manifold $M$, and $\Gamma$ be the copy of $\pi_{1}\Sigma$ in
$\mathrm{Isom}(\hy^{3})$ induced by the immersion $f$. Suppose $\Gamma$
is convex cocompact, then there are explicit constants $C_{1}(\Sigma,M)$and
$C_{2}(\Sigma,M)$ not bigger than 1 such that 
\[
C_{1}(\Sigma,M)\cdot\delta_{\Gamma}\leq h(\Sigma)\leq C_{2}(\Sigma,M)\cdot\delta_{\Gamma}.
\]
Moreover, each equality holds if and only if the marked length spectrum
of $\Sigma$ is proportional to the marked length spectrum of $M$,
and the proportion is the ration $\frac{\delta_{\Gamma}}{h(\Sigma)}.$ 
\end{cor}

\subsection{Minimal hyperbolic germs}

\subsubsection{Minimal hyperbolic germs}

\textsf{In the next three subsections, following Uhlenbeck's assumptions
in \cite{Uhlenbeck:1983wl}, we shall assume $\Sigma$ is a closed
surface. }

\textsf{Recall that  $\mathcal{H}$ is the set of the isotopy classes
of pairs consisting of a Riemann metric $g$ and a symmetric 2-tensor
$B$ on $\Sigma$ such that the trace of $B$ w.r.t. $g$ is zero
and $(g,B)$ satisfies the Gauss-Codazzi equations (cf. Remark \ref{rem: Gauss-Codazzi +miminal}
). Such a pair $(g,B)\in\mathcal{H}$ can be integrated to an immersed
minimal surfaces of a hyperbolic 3-manifold with the induced metric
$g$ and second fundamental form $B$. Moreover, we know that for
each data $(g,B)\in\mathcal{H}$ there exists a representation }$\rho:\pi_{1}(\Sigma)\to\mbox{Isom }(\hy{}^{3})\cong\psl(2,\co)$\textsf{
leaving this minimal immersion invariant.Thus, we have a map 
\begin{equation}
\Phi:\mathcal{H}\to\mathcal{R}(\pi_{1}(\Sigma),\psl(2,\co)),\label{eq: charater variality}
\end{equation}
where }$\mathcal{R}(\pi_{1}(\Sigma),\psl(2,\co))$\textsf{ is the
space of conjugacy classes of representations of $\pi_{1}(S)$ into
}$\psl(2,\co)$.

\textsf{The following corollary is an obvious consequence of Theorem
\ref{Thm: main-1}. Recall that $h(g,B)$ denotes the topological
entropy of the geodesic flow for the immersed surface $(\Sigma,g)$
with second fundamental form $B$.}

\begin{cor}
\label{Cor: taubes' space} Let $\rho\in\mathcal{R}(\pi_{1}(\Sigma),\mathrm{PSL}(2,\mathbb{C}))$
be a discrete, convex cocompact representation and suppose $(g,B)\in\Phi^{-1}(\rho)\neq\emptyset$.
Then there are explicit constants $C_{1}(g,B)$and $C_{2}(g,B)$ not
bigger than 1 such that 
\[
C_{1}(g,B)\cdot\delta_{\rho(\pi_{1}\Sigma)}\leq h(g,B)\leq C_{2}(g,B)\cdot\delta_{\rho(\pi_{1}\Sigma)}\leq\delta_{\rho(\pi_{1}\Sigma)}
\]
with the last equality if and only if $B$ is identically zero which
holds if and only if $\rho$ is Fuchsian. \end{cor}
\begin{proof}
$(g,B)\in\Phi^{-1}(\rho)$ means that there exists a $\pi_{1}$--injective
immersion $f:\Sigma\to\rho(\pi_{1}\Sigma)\backslash\hy^{3}=M$ such
that the induced metric is $g$ and the second fundamental form is
$B$, where $(M,h)$ is a convex cocompact hyperbolic 3--manifold.
Therefore, by Theorem \textsf{\ref{Thm: main-1}} we have\textsf{
}
\[
C_{1}(\Sigma,M)\cdot\delta_{\rho(\pi_{1}\Sigma)}\leq h(\Sigma)\leq C_{2}(\Sigma,M)\cdot\delta_{\rho(\pi_{1}\Sigma)}.
\]
Then we pick $C_{1}(g,B)=C_{1}(\Sigma,M)$ and $C_{2}(g,B)=C_{2}(\Sigma,M)$.
The rightmost inequality is because $C_{2}(g,B)=C_{2}(\Sigma,M)\leq1$,
and the rigidity is the consequence of Corollary \ref{cor:rigidity}.
\end{proof}

\begin{rem}
By Sullivan's theorem, we know that $\delta_{\rho(\pi_{1}\Sigma)}=\dim_{H}\Lambda(\rho(\pi_{1}\Sigma))$.
Thus, we can replace the critical exponent by the Hausdorff dimension
in the above corollary.
\end{rem}

\subsubsection{Quasi-Fuchsian Spaces}

\textsf{We call a discrete faithful representation} $\rho:\pi_{1}(\Sigma)\to\mbox{Isom }(\hy^{3})$
\textit{quasi-Fuchsian} \textsf{if and only if the limit set $\Lambda(\rho(\pi_{1}\Sigma))$
of $\rho(\pi_{1}\Sigma)$ is a Jordan curve and the domain of discontinuity
$\partial_{\infty}\hy^{3}\backslash\Lambda(\rho(\pi_{1}\Sigma))$
is composed by two invariant, connected, simply-connected components.
$\mathcal{QF}$ denotes the set of quasi-Fuchsian representations.}

\textsf{We notice that if $\rho\in\mathcal{QF}$, elements in $\Phi^{-1}(\rho)$
are $\pi_{1}(\Sigma)-$injective minimal immersions from $\Sigma$
to $\rho(\pi_{1}(\Sigma))\backslash\hy^{3}$. Moreover, Uhlenbeck
in \cite{Uhlenbeck:1983wl} points out that for $\rho\in\mathcal{QF}$,
$\Phi^{-1}(\rho)$ is always a non-empty set. }

\begin{cor}
\label{cor:(main thm app AF)}Let $\rho\in\mathcal{QF}$ be a quasi-Fuchsian
representation and $(g,B)\in\Phi^{-1}(\rho)$. Then there are explicit
constants $C_{1}(g,B)$ and $C_{2}(g,B)$ are not bigger than $1$
such that 
\[
C_{1}(g,B)\cdot\delta_{\rho(\pi_{1}\Sigma)}\leq h(g,B)\leq C_{2}(g,B)\cdot\delta_{\rho(\pi_{1}\Sigma)}\leq\delta_{\rho(\pi_{1}\Sigma)}
\]
with the last equality if and only if $B$ is identically zero which
holds if and only if $\rho$ is Fuchsian.
\end{cor}

\textsf{Using the above corollary, we can give another proof the famous
Bowen's rigidity theorem. }

\begin{cor}
[Bowen's rigidity \cite{Bowen:1979eh}]\label{cor: Bowen's rigidity}A
quasi-Fuchsian representation $\rho\in\mathcal{QF}$ is Fuchsian if
and only if $\dim_{H}\Lambda_{\Gamma}=1$.
\end{cor}

\begin{proof}
For any $(g,B)\in\Phi^{-1}(\rho)$, we have $\Sigma$ is an immersed
minimal surface in a quasi-Fuchsian manifold $M=\rho(\pi_{1}\Sigma)\backslash\hy^{3}$
with the induced metric $g$ and the second fundamental form $B$.
Let $K(\Sigma)$ denote the Gaussian curvature of $\Sigma$ in $M$,
then by the Gauss-Codazzi equation $K(\Sigma)\leq-1$. Therefore using
the Theorem B in \cite{Katok:1982jz}, we have 
\[
h(g,B)\geq\left(\frac{-\int K(\Sigma)\mbox{d}A}{\mbox{Area}(\Sigma)}\right)^{\frac{1}{2}}\geq1.
\]

Hence the result is derived by the above lower bound of $h(\Sigma)$
plus the above corollary.
\end{proof}

\subsubsection{Almost-Fuchsian spaces}

\textsf{Recall that the }\textit{almost-Fuchsian}\textsf{ space $\mathcal{AF}$
is the space of minimal hyperbolic germs $(g,B)\in\mathcal{H}$ such
that} $\left\Vert B\right\Vert _{g}<2$. \textsf{Given a hyperbolic
metric $m\in\mathcal{F}$ and a holomorphic quadratic differential
$\alpha\in Q([m])$, we discuss an informative smooth path 
\[
r(t)=(g_{t},tB)\subset\mathcal{AF},
\]
where $g_{t}=e^{2u_{t}}m$ and $B=\mathrm{Re}(\alpha)$ satisfying
}$\left\Vert tB\right\Vert _{g_{t}}^{2}<2$ \textsf{. Notice that
$u_{t}:\Sigma\to\real$ is well-defined and smooth on $t$ (cf. Theorem
\ref{Thm: Uhlenbeck solution curve}).}

\textsf{Through studying this path, we can learn many geometric features
of the Fuchsian space $\mathcal{F}$. First, we will see how the entropy
behaves while we change the data along the ray $r(t)$ in $\mathcal{AF}$.
To be more precise, in the following we denote the unit tangent bundle
of $\Sigma$ associated with the data $r(t)$ by $T^{1}\Sigma_{r(t)}$.}

\textsf{We recover the following theorem by employing the reparametrization
method. }

\begin{thm}
[Sanders, Theorem 3.5 \cite{Sanders:2014wv} ]\label{Thm: entropy decreasing}
Consider the entropy function restricting on the almost-Fuchsian space
$h:\mathcal{AF}\to\real$, then \begin{itemize}

\item[1.] the entropy function $h$ realizes its minimum at the Fuchsian
space $\mathcal{F}$, and

\item[2.] for $(m,0)\in\mathcal{F}$, $h$ is monotone increasing
along the ray $r(t)=(g_{t},tB)$ provided $\left\Vert tB\right\Vert _{g_{t}}<2$,
i.e. $r(t)\subset\mathcal{AF}$, where $g_{t}=e^{2u_{t}}m$.

\end{itemize}
\end{thm}

\textsf{Fixed $t_{0}>0$, $r(t_{0})=(e^{2u_{t_{0}}}m,t_{0}B)$ defines
the geodesic flow $\phi^{t_{0}}:T^{1}\Sigma_{r(t_{0})}\to T^{1}\Sigma_{r(t_{0})}$.
For any $t>t_{0}$, we want to show $h(t_{0},g_{0})\leq h(t,g)$ and
equality holds if and only if $B=0$ . Since $(\Sigma,g_{t})$ is
negatively curved, using the distance function $d_{g_{t}}$, we can
construct the Busemann cocycle $B_{\xi}^{g_{t}}(x,y)$ like we did
in the proof of Theorem \ref{Thm: main-1}. Then by Theorem \ref{Thm: Sambarino},
we can reparametrize the geodesic flow $\phi^{t_{0}}$ induced by
the data $r(t_{0})=(e^{2u_{t_{0}}}m,t_{0}B)$ on $T^{1}\Sigma_{r(t_{0})}$
by a Hölder function $F_{t}$ on $T^{1}\Sigma_{(t_{0},g_{0})}$ such
that $\phi^{t}:T^{1}\Sigma_{r(t)}\to T^{1}\Sigma_{r(t)}$ is conjugated
to $(\phi^{t_{0}})^{F_{t}}:T^{1}\Sigma_{r(t_{0})}\to T^{1}\Sigma_{r(t_{0})}$
. We consider the pressure $P_{\phi^{t_{0}}}:C(T^{1}\Sigma_{r(t_{0})})\to\real$,
and we have} 
\begin{alignat*}{1}
P_{\phi^{t_{0}}}(-h_{F_{t}}\cdot F_{t}) & =0=P_{\phi^{t_{0}}}(-h(g_{t},tB)\cdot F_{t})\\
P_{\phi^{t_{0}}}(-h_{F_{t_{0}}}\cdot1) & =0=P_{\phi^{t_{0}}}(-h(g_{t_{0}},t_{0}B)\cdot1).
\end{alignat*}

\begin{rem}
Without using Theorem \ref{Thm: Sambarino} , when $t_{0}$ and $t$
are small, the structure stability of Anosov flows also gives us the
reparametrizing function $F_{t}$. We will see more details about
this perspective in the next subsection. 
\end{rem}
\begin{proof}[Proof of the theorem \ref{Thm: entropy decreasing}]
Because almost-Fuchsian is quasi-Fuchsian, the first assertion is
a consequence of Corollary \ref{cor:(main thm app AF)} and the fact
that $h(g,B)\geq1$ (cf. the proof of Corollary \ref{cor: Bowen's rigidity}).
The remaining part of this proof is devoted to the second assertion.

It's enough to show $l_{g_{t}}(\gamma)\leq l_{g_{t_{0}}}(\gamma)$,
for $t>t_{0}$ and $\forall[\gamma]\in[\pi_{1}\Sigma]$. Because if
$l_{g_{t}}(\gamma)\leq l_{g_{t_{0}}}(\gamma)$ for all $[\gamma]\in[\pi_{1}\Sigma]$,
then we have $F_{t}$ is cohomologous to a function which is bigger
then $F_{t_{0}}$$(=1)$, abusing the notation, we denote this function
by $F_{t}$ again, i.e. $l_{g_{t}}(\gamma)\leq l_{g_{t_{0}}}(\gamma)\implies F_{t}\leq F_{t_{0}}=1$.

\textbf{Claim: }$F_{t}\leq F_{t_{0}}=1$$\implies$$h_{t}\geq h_{t_{0}}$
for $t\leq t_{0}$ and the equality holds if and only if $F_{t}\sim1$.

pf. we repeat the same trick used in the proof of Theorem \ref{Thm: main-1}.
Since the pressure is monotone ( see Remark \ref{rem: pressure} ),
we have

\[
0\leq F_{t}\leq F_{t_{0}}\implies0\geq-h_{t}F_{t}\geq-h_{t}F_{t_{0}}\implies P_{\phi^{t_{0}}}(-h_{t}F_{t})\geq P(-h_{t}F_{t_{0}}).
\]
Thus, 
\[
0=P_{\phi^{t_{0}}}(-h_{t_{0}})=P_{\phi^{t_{0}}}(-h_{t_{0}}F_{t_{0}})=P_{\phi^{t_{0}}}(-h_{t}F_{t})\geq P_{\phi^{t_{0}}}(-h_{t}F_{t_{0}})=P_{\phi^{t_{0}}}(-h_{t}).
\]

By definition, 
\[
0=P_{\phi^{t_{0}}}(-h_{t_{0}})=\sup_{\nu\in\mathcal{M}^{\phi^{t_{0}}}}h(\nu)+\int(-h_{t_{0}})\mathrm{d}\nu\implies\sup_{\nu\in\mathcal{M}^{\phi^{t_{0}}}}h(\nu)=h(\mu_{h_{t_{0}}})=h_{t_{0}}
\]
where $\mu_{h_{t_{0}}}$ is the equilibrium state of the function
$-h_{t_{0}}\cdot1$ and 
\[
0\geq P_{\phi^{t_{0}}}(-h_{t})=\sup_{\nu\in\mathcal{M}^{\phi^{t_{0}}}}h(\nu)+\int(-h_{t})\mathrm{d}\nu=h_{t_{0}}-h_{t}.
\]
To see the equality case, we notice that $h_{t_{0}}=h_{t}$ implies
that $h(\mu_{h_{t}F_{t}})=h_{t}=h_{t_{0}}$, i.e. 
\[
0=P_{\phi^{t_{0}}}(-h_{t_{0}}\cdot1)=h(\mu_{h_{t}F_{t}})-\int h_{t_{0}}\cdot1\mathrm{d}\mu_{h_{t}F_{t}}.
\]
In other words, $\mu_{F_{t}h_{t}}$ is the equilibrium state of the
constant function $-h_{t_{0}}\cdot1$. Hence, by the uniqueness of
equilibrium state (cf. Theorem \ref{Thm: unique eq state}) we have
$\mu_{h_{t}F_{t}}=\mu_{h_{0}\cdot1}$ and which implies $h_{t}F_{t}\sim h_{0}\cdot1$,
i.e. $F_{t}\sim1$. We now complete the proof of this claim.\smallskip{}

By the above claim, we know the equality holds if and only if $F_{t}\sim1$,
i.e. $l_{g_{0}}(\gamma)=l_{g}(\gamma)$ for all $[\gamma]\in[\pi_{1}\Sigma]$.
By the marked length spectrum theorem \cite{Otal:1990gu}, this means
that $g_{0}=g_{1}$. In other words, $h(g_{t},tB)$ is strictly increasing
when $u_{t}\neq0$. 

To prove $l_{g_{t}}(\gamma)\leq l_{g_{t_{0}}}(\gamma)$ for $t>t_{0}$,
it is enough to show that $g_{t}=e^{2u_{t}}m$ is decreasing, i.e.
$\frac{d}{dt}u_{t}<0$ for all $t>0$. 

Because fixing a free homology class of a closed curve $\tau$ and
let $c_{t}$ denote the closed geodesic in this class under the metric
$g_{t}$, assuming that $g_{t}$ is decreasing, then we have $||\cdot||_{g_{t}}\leq||\cdot||_{g_{t_{0}}}$
for $t>t_{0}$. Thus, 
\[
l_{g_{t}}(c_{t})\leq l_{g_{t}}(c_{t_{0}})=\int_{c_{t_{0}}}||v||_{g_{t}}\leq\int_{c_{t_{0}}}||v||_{g_{t_{0}}}=l_{g_{t_{0}}}(c_{t_{0}}),
\]
where $v$ is the unit tangent vector of $c_{t_{0}}$ for the metric
$g_{t_{0}}$, i.e. $v(s):=\dfrac{d}{ds}(c_{t_{0}}(s))/||\dfrac{d}{ds}(c_{t_{0}}(s))||_{g_{t_{0}}}$.

\textbf{Claim:} $g_{t}=e^{2u_{t}}m$ is decreasing, i.e. $\frac{d}{dt}u_{t}<0$
for all $t>0$. 

pf. by Theorem \ref{thm: gauss equation} (the Gauss equation), we
have 
\begin{equation}
K_{g_{t}}=-1-\frac{1}{2}t^{2}e^{-4u_{t}}||B||_{m}^{2}=e^{-2u_{t}}(-\Delta_{m}u_{t}-1).\label{eq:curvature}
\end{equation}
Taking the derivative of equation \ref{eq:curvature} w.r.t. $t$
and evaluating at $t_{0}$ provided $\left\Vert t_{0}B\right\Vert _{g_{0}}^{2}<2$
, then 
\[
-\Delta_{m}\dot{u}_{t_{0}}=e^{2u_{t_{0}}}\cdot\dot{u}_{t_{0}}(\left\Vert t_{0}B\right\Vert _{g_{0}}^{2}-2)-t_{0}e^{-2u_{t_{0}}}\left\Vert B\right\Vert _{m}^{2}.
\]
 Because for the fixed $t_{0}$, at a maximum of $\dot{u_{t_{0}}}$
we have $-\Delta_{m}\dot{u}_{t_{0}}\geq0.$ Thus 
\[
e^{2u_{t_{0}}}\cdot\dot{u}_{t_{0}}(\left\Vert t_{0}B\right\Vert _{g_{0}}^{2}-2)-t_{0}e^{-2u_{0}}\left\Vert B\right\Vert _{m}^{2}\geq0.
\]
The hypothesis $\left\Vert t_{0}B\right\Vert _{g_{0}}^{2}<2$ implies
that $\dot{u}_{t_{0}}\leq0$ at each maximum; hence $\dot{u}_{t_{0}}\leq0$
for all points on $\Sigma$. Moreover, if $\dot{u}_{t_{0}}=0$ for
some $t_{0}>0$, then we have $B=0$. This means $u_{t}\equiv0$ for
all $t$.

\end{proof}

\begin{rem}
The main difference between our proof and Sanders' proof in \cite{Sanders:2014wv}
is that in \cite{Sanders:2014wv} he used a sophisticated formula
of the derivative of the topological entropy whereas in our reasoning
we examine the length changing along the path directly.
\end{rem}

\subsubsection{Another metric on $\mathcal{F}$}

\textsf{Following the previous theorem, Sanders proves that we can
define a metric on the Fuchsian space $\mathcal{F\subset\mathcal{H}}$
by taking the Hessian of the topological entropy along the path $r(s)=(e^{2u_{s}}m,sB)$.
The striking point is that this metric is bounded below by the Weil-Petersson
metric on $\mathcal{F}$. }

\textsf{Recall that the fiber of the cotangent bundle of $m\in\mathcal{F}$
is identified with the space of holomorphic quadratic differentials
on the Riemann surface $(\Sigma,m)$. Thus, in order to connect the
Hessian of the entropy with the Weil-Petersson metric, we will prove
that the Hessian of the topological entropy along the given path $r(s)$
gives us a way to measure holomorphic quadratic differentials on $(\Sigma,m).$
It is because $r(s)$ is defined by the data $(m,B)$, where $B$
is given by a holomorphic quadratic differential $\alpha$ such that
$B=\mathcal{\mathrm{Re}}(\alpha)$. }

\textsf{In the following we give another proof of Sanders' theorem
by using the pressure metric we introduced in the beginning of this
note. }

\textsf{Before we start proving this result, we recall several important
concepts of Anosov flows. We first notice that by the structure stability
of the Anosov flow (cf. Prop. 1 in \cite{Pollicott:1994wq} or \cite{delaLlave:1986ce}),
when $s$ is small, the geodesic flows $\phi^{s}:T^{1}\Sigma_{r(s)}\to T^{1}\Sigma_{r(s)}$
conjugates to the reparametrized flow $\phi^{F_{s}}:T^{1}\Sigma_{r(0)}\to T^{1}\Sigma_{r(0)}$
where $\phi:T^{1}\Sigma_{r(0)}\to T^{1}\Sigma_{r(0)}$ is the geodesic
flow on $T^{1}\Sigma_{r(0)}$ and $F_{s}$ the is the reparametrizing
Hölder continuous function. Since the path $r(s)$ is a smooth one
parameter family in $\mathcal{AF}$, the structure stability theorem
also indicates that $\{F_{s}\}$ form a smooth one parameter family
of Hölder continuous functions on $T^{1}\Sigma_{r(0)}$ .}

\textsf{Since we shall only be interested in metrics $g_{s}$ close
to $g_{0}(=m)$, it suffices to consider one parameter families given
by perturbation series: for $s$ small, 
\[
g_{s}=g_{0}+s\cdot\dot{g}_{0}+\frac{s^{2}}{2}\ddot{g}_{0}+...,\mbox{ and }F_{s}=F_{0}+s\cdot\dot{F}_{0}+\frac{s^{2}}{2}\ddot{F}_{0}+...,
\]
where $\dot{g}_{0},\ddot{g}_{0}$,... are symmetric 2-tensors on $T^{1}\Sigma_{r(0)}$
and $\dot{F}_{0}$, $\ddot{F}_{0}$,... are Hölder continuous functions
on $T^{1}\Sigma_{m}$. }

\textsf{The following lemma connects the derivatives of $g_{s}$ with
$F_{s}$.}

\begin{lem}
[Pollicott, Lemma 5 \cite{Pollicott:1994wq}]\label{lemma: pollicott trick}
\begin{equation}
\int_{T^{1}\Sigma}\dot{F}_{0}\mbox{d}\mu_{0}=\int_{T^{1}\Sigma}\dot{g}_{0}(v,v)\mbox{d}\mu_{0},\label{eq: pollicott lemma up}
\end{equation}
 and 
\begin{equation}
{\displaystyle \int_{T^{1}\Sigma}\ddot{F}_{0}\mbox{d}\mu_{0}}\leq\int_{T^{1}\Sigma}\frac{\ddot{g}_{0}(v,v)}{2}\mbox{d}\mu_{0}.\label{eq:pollicott lemma down}
\end{equation}

\end{lem}

\begin{rem*}
The proof of above lemma is a straightforward computation. However,
in the sake of brevity we omit the proof. 
\end{rem*}

\textsf{The following lemma reveals the relation between Weil-Petersson
metric and the second derivative of the metric $g_{s}$.}

\begin{lem}
\label{lemma: liouville is bm}

\[
\int_{T^{1}\Sigma}\frac{\ddot{g}_{0}(v,v)}{2}\mbox{d}\mu_{0}=\int_{T^{1}\Sigma}\ddot{u}_{0}\mbox{d}\mu_{0}=-2\pi\int_{\Sigma}\left\Vert \alpha\right\Vert _{m}^{2}\mbox{d}\mbox{V}_{m},
\]

\end{lem}

\begin{proof}
\textbf{Claim:} 
\[
\int\ddot{u}_{0}\mbox{d}\mu_{0}=-2\pi\int_{\Sigma}\left\Vert \alpha\right\Vert _{m}^{2}\mbox{d}\mbox{V}_{m}.
\]

pf. for $m\in\mathcal{F}$ and $\alpha\in Q([m])$, we notice that
the Gauss equation (Theorem \ref{thm: gauss equation}) for this given
data $r(s)=(e^{2u_{s}}m,s\cdot\mathrm{Re}\alpha)$ is that 
\[
\Delta_{m}u_{s}+1-e^{-2u_{s}}-se^{-2u_{s}}||\alpha||_{m}^{2}=0,
\]
where $||\alpha||_{m}$ is the norm of $\alpha$ with respect to the
hyperbolic metric $m$.

Taking the derivative with respect to $s$, we have 
\[
-\Delta_{m}\dot{u}_{s}=e^{2u_{s}}\cdot\dot{u}_{s}(\left\Vert s\alpha\right\Vert _{m}^{2}-2)-se^{-2u_{s}}\left\Vert \alpha\right\Vert _{m}^{2}.
\]
The maximum principle implies that $\dot{u}_{0}=0$. We differentiate
the above equation again and evaluate at $s=0$, then we get 
\begin{equation}
-\Delta_{m}\ddot{u}_{0}=-2\ddot{u}_{0}-2\left\Vert \alpha\right\Vert _{m}^{2}.\label{eq:gass apply}
\end{equation}
By integrating the equation (\ref{eq:gass apply}) with respect to
the volume form of $m$, and because $\Sigma$ is has no boundary
we have 
\[
0=-2\int_{\Sigma}\ddot{u}_{0}\mbox{d}V_{m}-2\int_{\Sigma}\left\Vert \alpha\right\Vert _{m}^{2}\mbox{d}V_{m}.
\]
Because in this case the Bowen-Margulis measure $\mu_{0}$ of the
geodesic flow $\phi:T^{1}\Sigma_{m}\to T^{1}\Sigma_{m}$ is the Liouville
measure, we have 

\[
2\pi\int_{\Sigma}\ddot{u}_{0}\mbox{d}V_{m}=\int_{T^{1}\Sigma}\ddot{u}_{0}\mbox{d}\mu_{0}.
\]

\medskip{}

\textbf{Claim:} 
\[
\int_{T^{1}\Sigma}\frac{\ddot{g}_{0}(v,v)}{2}\mbox{d}\mu_{0}=\int_{T^{1}\Sigma}\ddot{u}_{0}\mbox{d}\mu_{0}.
\]

pf. It is straightforward, because $\dot{u}_{0}=0$ and $\ddot{g}_{0}=2\ddot{u}_{0}m$.
\end{proof}

\textsf{Now we are ready to state and prove the main result of this
subsection. }
\begin{thm}
[Sanders, Theorem 3.8 \cite{Sanders:2014wv} ]\label{thm:hessian entropy}One
can define a Riemannian metric on the Fuchsian space $\mathcal{F}$
by using the Hessian of $h$. Moreover, this metric is bounded below
by $2\pi$ times the Weil-Petersson metric on $\mathcal{F}$.
\end{thm}

\begin{proof}
[Proof of the Theorem \ref{thm:hessian entropy}]Here we consider
$c_{t}:=-h(g_{t},tB)\cdot F_{t}=-h_{t}F_{t}$. Since $P_{\phi}(c_{t})=0$,
we know that $\int\dot{c}_{0}\mbox{d}\mu_{0}=0$ where $\mu_{0}$
is the Bowen-Margulis measure of the flow $\phi:T^{1}\Sigma_{r(0)}\to T^{1}\Sigma_{r(0)}$,
i.e. $m_{c_{0}}=\mu_{0}$ and $\dot{c}_{0}\in T_{c_{0}}\mathcal{P}(T^{1}\Sigma_{m})$.
Therefore, by Proposition \ref{prop: pressure metric}, the pressure
metric of $\dot{c}_{0}$ is 
\[
\left\Vert \dot{c}_{0}\right\Vert _{P}^{2}=-\frac{\mathrm{Var}(\dot{c}_{0},\mu_{0})}{\int c_{0}\mbox{d}\mu_{0}}=\frac{\int\ddot{c}_{0}\mbox{d}\mu_{0}}{\int c_{0}\mbox{d}\mu_{0}}.
\]
Notice that $h_{0}=1$, $F_{0}=1$, and $\dot{u}_{0}=0$, so by Lemma
\ref{lemma: pollicott trick} 
\[
\int\dot{F}{\rm d}\mu_{0}=\int\dot{g}_{0}{\rm d}\mu_{0}=\int2\dot{u}_{0}m(v,v)\mathrm{d}\mu_{0}=0,
\]
and hence 
\begin{alignat*}{1}
0\leq\left\Vert \dot{c}_{0}\right\Vert _{P}^{2} & =\ddot{h}_{0}+2\dot{h}_{0}\int\dot{F}_{0}\mbox{d}\mu_{0}+h_{0}\int\ddot{F}_{0}\mbox{d}\mu_{0}\\
 & =\ddot{h}_{0}+\int\ddot{F}_{0}\mbox{d}\mu_{0}.
\end{alignat*}

Therefore, 
\begin{alignat*}{2}
\ddot{h}_{0}\geq & -{\displaystyle \int_{T^{1}\Sigma}\ddot{F}_{0}\mathrm{d}\mu_{0}}\\
\geq & -\int_{T^{1}\Sigma}\frac{\ddot{g}_{0}(v,v)}{2}\mbox{d}\mu_{0} & Lemma\mbox{ }\ref{lemma: pollicott trick}\\
= & -\int_{T^{1}\Sigma}\ddot{u}_{0}m(v,v)\mbox{d}\mu_{0} & \quad Lemma\mbox{ }\ref{lemma: liouville is bm}\\
= & -\int_{T^{1}\Sigma}\ddot{u}_{0}\mbox{d}\mu_{0}\\
= & 2\pi\int_{\Sigma}\left\Vert \alpha\right\Vert _{m}^{2}\mbox{d}\mbox{V}_{m} & Lemma\mbox{ }\ref{lemma: liouville is bm}\\
= & 2\pi\left\Vert \alpha\right\Vert _{WP}^{2}.
\end{alignat*}

\end{proof}

\begin{rem}
Comparing with Bridgeman's results in \cite{Bridgeman:2010kt}, we
suspect: \begin{itemize}

\item[1.] $\int_{T^{1}\Sigma}\ddot{F}_{0}\mbox{d}\mu_{0}=\int_{T^{1}\Sigma}\frac{\ddot{g}_{0}(v,v)}{2}\mbox{d}\mu_{0}$,
and

\bigskip{}

\item[2.] $\left\Vert \dot{c}_{0}\right\Vert _{P}=0$.

\end{itemize}
\end{rem}

\bibliographystyle{amsalpha}
\bibliography{/Users/nyima/Dropbox/TEX/Bibtex/Papers3_backup_20151007,/Users/nyima/Dropbox/TEX/Bibtex/Papers3_backup_20151006}

\end{document}